\documentclass[a4paper,11pt]{amsart}
\usepackage{cases, cite}
\usepackage{amsfonts,latexsym,rawfonts,amsmath,amssymb,amsthm}
\usepackage[plainpages=false]{hyperref}
\usepackage{graphicx}
\usepackage{setspace}

\RequirePackage{color}

\theoremstyle{plain}
  \newtheorem{theorem}{Theorem}[section]
  \newtheorem{lemma}[theorem]{Lemma}
  \newtheorem{proposition}[theorem]{Proposition}
  \newtheorem{conjecture}[theorem]{Conjecture}
  \newtheorem{corollary}[theorem]{Corollary}

\theoremstyle{remark} 
  \newtheorem*{remark}{Remark}
  \newtheorem*{definition}{Definition}

\theoremstyle{definition}

\numberwithin{equation}{section}

\DeclareMathOperator*{\argmin}{argmin}
\DeclareMathOperator*{\argmax}{argmax}
\DeclareMathOperator{\vol}{vol}
\DeclareMathOperator*{\intr}{int}
\DeclareMathOperator*{\Tr}{tr}
\DeclareMathOperator{\Sym}{Sym}
\DeclareMathOperator{\im}{Im}
\DeclareMathOperator{\conv}{conv}
\DeclareMathOperator{\Aut}{Aut}

\newcommand{\R}{\mathbb R}

\newtheorem{manualtheoreminner}{Theorem}
\newenvironment{manualtheorem}[1]{%
	\renewcommand\themanualtheoreminner{#1}%
	\manualtheoreminner
}{\endmanualtheoreminner}

\newtheorem{manualpropinner}{Proposition}
\newenvironment{manualprop}[1]{%
	\renewcommand\themanualpropinner{#1}%
	\manualpropinner
}{\endmanualpropinner}

\begin{document}

\title{Some new positions of maximal volume of convex bodies}

\author{Shiri Artstein-Avidan}
\address{School of Mathematical Sciences, Tel Aviv University, Tel Aviv, 66978, Israel.}
\email{shiri@tauex.tau.ac.il}
\thanks{Supported in part by the 
	European Research Council (ERC) under the European Union’s Horizon 2020
	research and innovation programme (grant agreement No 770127), and in part by
	ISF grant 784/20}
\author{Eli Putterman}
\email{putterman@mail.tau.ac.il}



\date{\today}



\begin{abstract}
In this paper, we extend and generalize several previous works on maximal-volume positions of convex bodies. First, we analyze the maximal positive-definite image of one convex body inside another, and the resulting decomposition of the identity. We discuss continuity and differentiability of the mapping associating a body with its positive John position. We then introduce the saddle-John position of one body inside another, proving that it shares some of the properties possessed by the position of maximal volume, and explain how this can be used to improve volume ratio estimates. We investigate several examples in detail and compare these positions. Finally, we discuss the maximal intersection position of one body with respect to another, and show the existence of a natural decomposition of identity associated to this position, extending previous work which treated the case when one of the bodies is the Euclidean ball. 
\end{abstract}

\maketitle

\baselineskip=16.4pt
\parskip=3pt

\section{Introduction}

The theme of the present paper is the investigation of positions of convex bodies in $\R^n$ which
are relatives of the well known John and L\"owner positions. The John
position of a convex body $K$ is defined as the affine image $K_J$ of $K$ 
for which the ellipsoid of maximal volumes contained in $K_J$ is the unit
Euclidean ball, while the L\"owner position of $K$ is its affine image $K_L$
such that the ellipsoid of minimum volume containing $K_L$ is the unit Euclidean ball. 

John's theorem from 1948 \cite{J} states that if $K$ is in John position, the intersection of its boundary with $S^{n - 1}$ supports an isotropic measure in a sense to be made precise shortly (see the statement of Theorem \ref{gen_john_thm} and the paragraph below); the converse statement was proven many years later by Ball \cite{B}.

Replacing the Euclidean ball by some other convex body, one may define
the maximal-volume image of a convex body $L$ inside a convex body $K$ as the
affine image of $L$ contained in $K$ with maximal volume among all such images; 
we say that $L$ is in maximal volume position inside $K$ if $L$ is its own maximal-volume image inside $K$.
This position was studied by several authors, among them Giannopoulos, Perissinaki and Tsolomitis \cite{GPT}, Bastero and Romance \cite{BR}, Gordon, Litvak, Meyer and Pajor \cite{GLMP}, and Gruber and Schuster \cite{GS}.

Denote by $\mathcal K^n$ the set of convex bodies in $\mathbb R^n$, $\mathcal K^n_o, \mathcal K^n_s \subset \mathcal K^n$ the set of convex bodies with the origin in their interior and the set of centrally symmetric convex bodies, respectively. It is well-known that when $L$ is in maximal
volume position inside $K$ then there is a decomposition of the identity
supported on contact pairs of $K$ and $L$ (for the definition of a contact pair, see \S \ref{section:pos_john}). More precisely, we have 

\begin{theorem}\label{gen_john_thm}
	Let $L \in \mathcal K^n$, $K \in \mathcal K^n_o$ with $K \subset L$, and suppose $L$ has maximal volume among all affine images of $L$ contained in $K$. 
	Then there exist contact pairs $(x_1, y_1), \ldots, (x_m, y_m)$  of $K, L$, that is,  $x_i \in \partial L \cap \partial K$, $y_i \in \partial L^\circ \cap \partial K^\circ$, and $\langle x_i, y_i\rangle = 1$, and constants $c_1,\ldots,c_m > 0$ such that:
	\begin{enumerate}
		\item $\sum_{i = 1}^m c_i y_i = 0$.
		\item $\sum_{i = 1}^m c_i (x_i \otimes y_i) = I_n$.
	\end{enumerate}
\end{theorem}

Here $x\otimes y$ stands for the rank one matrix $xy^T$. 

In the classical John's theorem, in which $L$ is taken to be the Euclidean ball, the contact pairs are of the form $(x_i, x_i)$ and one says that the measure giving mass $c_i$ to the point $x_i$ on the sphere is isotropic (see \cite[Section 2.1]{AAGM}). The fact that solutions of extremal problems give rise to isotropic measures in many different forms is well-known, see e.g. \cite{GM}. 

As we stated, it was proven by Ball that when $L = B^n_2$ and $L \subset K$, i.e., in the setting of the classical John theorem, the existence of an isotropic measure supported on contact points is not only implied by, but also implies that $K$ is in John position. (The same equivalence holds for L\"owner position if we assume $K \subset B^n_2$.) This is no longer the case in the general setting of pairs of bodies, as can be seen even by two-dimensional examples such as $B(\ell_1^n) \subset B(\ell_\infty^n)$. However, one does obtain an ``if and only if'' characterization of the position by the existence of a decomposition of the identity when considering a modification of the above position, namely the \textit{positive John position}:

\begin{definition}
	Let $L, K$ be convex bodies with nonempty interior. We define a positive image of $L$ in $K$ to be a set of the form $PL + z$ contained in $K$, with $z \in \mathbb R^n$ and $P$ a positive-definite matrix. We say that {\em $L$ is in positive John position in $K$}, or that {\em $K$ is in positive John position with respect to $L$}, if $L \subset K$ and $L$ has maximal volume among all positive images of $L$ contained in $K$.
\end{definition}

Unlike the position of maximal volume, the positive John position is easily shown to be unique (see Proposition \ref{pos_john_ex}). Also note that when $L$ is a Euclidean ball, positive John position with respect to $L$ is precisely the usual John position.
 
The analogue of John's theorem for positive John position characterizes the positive John position of $K$ with respect to $L$ in terms of contact pairs:

\begin{theorem}\label{pos_john_contact} Let $K \in \mathcal K^n_o$, $L \in \mathcal K^n$. Then $K$ is in positive John position with respect to $L$ if and only if $L \subset K$ and there exist contact pairs $(x_1, y_1), \ldots, (x_m, y_m)$  of $K, L$ and $c_1, \ldots, c_m > 0$ such that:
	\begin{align}
	0 &= \sum_{i = 1}^m c_i y_i,   \\
	I_n &= \sum_{i = 1}^m c_i (x_i \otimes y_i + y_i\otimes x_i).   
	\end{align}
\end{theorem}

This theorem has been proven by different methods in \cite{BR,GLMP}; for completeness, we also provide a proof. 

Note that the representation of the identity associated with the ``contact pairs" in
the above theorem differs from the classical one in Theorem \ref{gen_john_thm} by a
symmetrization, in a sense doubling the number of rank-one matrices which are used to obtain the identity.

Our first goal in the present work is to further investigate the positive John position. First of all, somewhat counterintuitively, we note that the positive John image of $L$ in $K$ might itself not be in a positive John position, due to the fact that a product of two positive-definite matrices may fail to be positive-definite. Instead, we have the following:

\begin{manualprop}{\ref{prop:pos_john_im}}
	For any convex bodies $K, L \in \mathcal K^n$, $L' = PL + z$ is the positive John image of $L$ in $K$ if and only if $P^{\frac{1}{2}} L + P^{-\frac{1}{2}} z$ is in positive John position inside $P^{-\frac{1}{2}} K$.
\end{manualprop}

Next, it seems that continuity and smoothness properties of the mapping associating a body with its John ellipsoid, or, more generally, associating a body with its positive John image within another body, have not previously been investigated in depth. We give a few results in this direction. First, letting  $\mathcal P^n$ denote the cone of positive-definite matrices, we show:

\begin{manualprop}{\ref{pos_john_contin}} The function $(P^*, z^*): \mathcal K^n \times \mathcal K^n \to \mathcal P^n \times \mathbb R^n$, defined such that $P^*(K, L) L + z^*(K, L)$ is the positive John image of $L$ in $K$, is continuous with respect to the Hausdorff metric on $\mathcal K^n$.
\end{manualprop}
 
We also show, in Theorem \ref{pos_john_sm}, that under some technical assumptions, fixing a polytope $L$ and a smooth $C^{k}$ body $K$, the mapping associating an orthogonal transformation $U$ to the the positive John image of the rotation $UL$ inside $K$, is $C^{k-1}$ smooth.

We define the \textit{positive John family} of $L$ in $K$ as the set of positive John positions of all orthogonal images of $L$ with respect to $K$. One distinguished member of this family is the (not necessarily unique) maximal affine position of $L$ inside $K$ from Theorem \ref{gen_john_thm}. We define a new distinguished position within this family, which we call \text{saddle-John position}, in a similar way: this is the positive John image of minimal volume among the positive John images of all $O_n$-images of $L$ inside $K$. (Again, this is not necessarily unique.) The classical position of maximal volume corresponds to the ``best" orthogonal image of $L$ inside $K$, so it and the saddle-John position are both extremal within the positive John family. This extremality is the key property which can be used to show that the saddle-John position shares with the classical maximal volume position the property of supporting a genuine decomposition of the identity on contact pairs (without taking the symmetric part), that is:

\begin{theorem}\label{saddle_contact}
	Let $K, L \subset \mathcal K^n$ (the set of convex bodies in $\mathbb R^n$ containing the origin in their interior), $L \in \mathcal K^n$, and suppose $L$ is in saddle-John position inside $K$. Then there exist contact pairs $(x_1, y_1), \ldots, (x_m, y_m)$  of $K, L$ and $c_1,\ldots,c_m > 0$ such that:
	\begin{enumerate}
		\item $\sum_{i = 1}^m c_i y_i = 0$.
		\item $\sum_{i = 1}^m c_i (x_i \otimes y_i) = I_n$.
	\end{enumerate} 
\end{theorem}

Gordon, Litvak, Meyer, and Pajor showed (settling a conjecture of Gr\"unbaum) that for any two convex bodies $K, L$, there exist affine images $K', L'$ respectively such that $L' \subset K' \subset -nL'$ \cite[Theorem 5.1]{GLMP}. In fact, the position of maximal volume of $L$ in $K$ satisfies this condition, after possibly translating $K$ and $L$. This result has applications to bounding the product of the volume ratios of two convex bodies, where the volume ratio $vr(K, L)$ is defined as the infimum of $\left(\frac{\vol(K)}{\vol(L')}\right)^{\frac{1}{n}}$ over all affine images $L'$ of $L$ contained in $K$; in particular, we immediately obtain $vr(K, L) vr(L, K) \le n$.

The crucial tool used in the proof of \cite[Theorem 5.1]{GLMP} is the existence of a decomposition of the identity supported on contact pairs of the position of maximal volume. We use the fact that saddle-John position yields a decomposition of the identity with precisely the same properties in order to show:

\begin{manualprop}{\ref{saddle_dilate}} Suppose that $L_s = PUL + z$ is the saddle-John image of $L$ inside $K$. Then there exists $a \in \mathbb R^n$ such that $K - a \subset -n(L_s - a)$.
\end{manualprop}

Since, for given $K, L$, the saddle-John image $L_{saddle}$ of $L$ in $K$ will usually have smaller volume than the maximal-volume image $L_{max}$ (and often much smaller volume), the upper bound on $vr(K, L)$ implied by Proposition \ref{saddle_dilate}, namely, $n\frac{\vol(L_{saddle})}{\vol(K)}$, may be much better than the bound on $vr(K, L)$ which follows from the inclusion $K \subset -nL_{max}$ proven by \cite{GLMP}. For instance, if we take $L = B^n_1$, $K = B^n_\infty$, the fact that $L$ is in saddle-John position in $K$ implies $K \subset n L$, which implies (given the fact that $\left(\frac{\vol(K)}{\vol(L)}\right)^{\frac{1}{n}} = \Theta(n)$) that $vr(K, L) \le c$ for $c$ a universal constant. On the other hand, the maximal volume position $L_{max}$ of $L$ has volume of order $\Theta(\sqrt n)^n \vol(L)$, so using the fact that $K \subset n L_{max}$ would give the much worse estimate $vr(K, L) \le O(\sqrt n)$. (For proofs of the above assertions regarding the position of maximal volume and the saddle-John position of $B^n_1$ and $B^n_\infty$, see \S \ref{subsec:saddle_examples}.)

The notion of the positive John family, and of saddle-John position in particular, raises many interesting questions, for instance: how different are the volumes of the saddle-John and the position of maximal volume of $L$ inside $K$? How does the positive John position of a random rotation of $K$ inside $L$ compare to the two extremal positive John positions? We begin to examine these questions for specific examples (the $\ell^1$- and $\ell^\infty$-balls, respectively), but even in these cases our understanding is not complete, and much room remains for further work.

Next, it turns out that when one of the bodies is an ellipsoid, as one may expect, the situation simplifies and all the elements in the positive John family of a given body have the same volume. While this fact is not particularly surprising, proving it does involve some careful analysis, and we show:

\begin{manualprop}{\ref{prop-PJPin-ellipsoid}} Let $P \in \mathcal P^n$ be a positive matrix, $E = P B^n_2$ the corresponding ellipsoid, and $L$ an arbitrary centrally symmetric convex body. Then $E$ is in positive John position with respect to $L$ if and only if $B^n_2$ is in positive John position with respect to $P^{-1} L$, i.e., $P^{-1} L$ is in L\"owner position. In particular, all the bodies in the positive John family of $L$ inside $E$ have the same volume.
\end{manualprop}

It seems likely that this property is unique to ellipsoids, although we leave this for future research.

In the final section of the paper, we switch gears and discuss the maximal intersection position of two convex bodies, which generalizes the position of maximal volume: we say that $K$ and $L$ are in \textit{maximal intersection position} if among all affine images of $L$ with the same volume as $L$, the one which has the largest intersection with $K$ is $L$ itself. The case where one of the bodies is a ball and the other is centrally symmetric was introduced and investigated by the first author and Katzin in \cite{AK}.

Like the John-type positions, these extremal positions in the case of a ball give rise to isotropic measures, and in fact quite explicit ones, given by the Lebesgue measure restricted to the part of the sphere contained in the intersection with $K$. For two general bodies, where we no longer assume one of the bodies is a ball, nor even that they are centrally symmetric, an analogue similar to Theorem \ref{gen_john_thm} holds. We show:

\begin{theorem}\label{gen_max_int_pos} Let $K, L \subset \mathbb R^n$ be convex bodies, and suppose that $K, L$ are in maximal intersection position and that $\vol_{n-1}(\partial K\cap\partial L) = 0$. For any $x \in \partial L$, let $\hat n_L(x)$ be the unit normal at $x$, which is defined $\mathcal H^{n - 1}$-almost everywhere on $\partial L$. Then we have
	\begin{align}
	\int_{K \cap \partial L} \hat n_L(x)\,d\mathcal H^{n - 1}           &= 0          , \\
	\int_{K \cap \partial L} x \otimes \hat n_L(x)\,d\mathcal H^{n - 1} &\propto I_n  .
	\end{align}
	The same formulae hold when interchanging the roles of $K, L$.
\end{theorem}

We give two different proofs, one which avoids approximation (which was a main tool in the previous work) and another which works for a much richer family of transformations (not just linear ones), but with more restrictive assumptions. We mention that a very recent manuscript \cite{BH} has generalized maximal intersection position in a different direction. 

One may combine the methods and results of the different parts of the paper, for instance by considering a ``saddle- maximal intersection position,'' or using set-valued analysis to study smoothness of these new families of positions. We make some remarks in this direction at the conclusion of the paper, and leave further combinations to dedicated readers. Additional results on smoothness of maximal intersection position, which require substantially different methods and hence have been left out of the present work, can be found in Chapter 4 of the second author's M. Sc. thesis \cite{P}. 

\subsection*{Organization}
The paper is organized as follows.
In Section \ref{sec:Prelim-and-LinAlg}, after providing background information and notation, we
state and prove some known and lesser-known facts from linear algebra
to be used throughout the text. Among these is a ``modified polar
decomposition" given in Lemma \ref{gen_pol_decomp}, which seems, as far as we can tell, not to have been previously noticed and may be of independent interest: for any fixed $M \in GL_n$, any $A \in GL_n$ has a unique representation as $A = PMU$ for a positive definite $P$ and an orthogonal $U$; moreover, the map $A \mapsto (P, U)$ is a diffeomorphism.

In Section \ref{section:pos_john} we discuss the positive John position of a convex body with respect to another, and prove the existence of the decomposition of the identity associated with contact pairs for this position; this differs from the genuine position of maximal volume by an extra symmetrization. We note a subtlety of this notion, arising from the fact that the product of two positive-definite matrices need not be positive-definite. In the remainder of the section, we define the positive John family of $L$ in $K$, prove that it is continuous for any pairs of two bodies, and show that it is differentiable under mild technical assumptions.

In Section \ref{section:saddle_john} we define the saddle-John
position, and show that it shares with the classical position of maximal volume the property of supporting a genuine decomposition of the identity on contact pairs, without taking the symmetric part. We also examine the disparity between positions of maximal volume and saddle-John positions for two pairs $K, L$: when $K = B^n_\infty$, $L = B^n_1$; when $K = B^n_1$ and $L = B^n_\infty$; and when $K = L = B^n_\infty$.
The example of the positive John family of a body $K$ inside an ellipsoid is analyzed in Section \ref{section:pos_john_ellips}; we show that in this case that given the positive John image $K'$ of $K$ itself, the positive John family of any orthogonal images of $K$ has the same volume as that of $K'$, and in fact may be obtained by a (somewhat unwieldy) formula given $K'$.
In Section \ref{section:max_int_pos} we develop the theory of maximal intersection
position of two bodies, including the representation of the identity corresponding to contact pairs on certain parts
of the intersection.   
Finally, we make several additional remarks on these families of positions.
 
\subsection*{Acknowledgments.}
Some of the results in this paper were obtained as part of the second author's thesis \cite{P}, which was carried out under the supervision of the first author at Tel Aviv University.

The authors would like to thank the anonymous referees for helpful comments. 

\section{Preliminaries and some linear algebra}\label{sec:Prelim-and-LinAlg}

\subsection{Notations and basic facts} 
We collect here the notation and basic facts in convex geometry we shall use. A comprehensive and up-to-date reference on the theory of convex bodies is the book of Schneider \cite{S}.

A convex body $K \subset \mathbb R^n$ is a compact convex set with nonempty interior. In this work, we shall assume for simplicity that all convex bodies we consider satisfy $0 \in \intr K$. $K$ is said to be centrally symmetric if $K = -K = \{-x: x \in K\}$. We write $\mathcal K^n_s \subset \mathcal K^n$ for the set of centrally symmetric convex bodies, and $\mathcal K^n_o$ for the set of convex bodies in $\mathbb R^n$ with the origin in their interior.

The support function $h_K: \mathbb R^n \to \mathbb R$ associated with the convex body $K$ is defined, for $u \in \mathbb R^n$, by
	\begin{equation}\label{supp}
		h_K(u) = \max \{\langle u, y\rangle: y \in K\}
	\end{equation}
The support function is convex and positively homogeneous of degree one, so it is completely determined by its restriction to the unit sphere $S^{n - 1}$.

The gauge function of $K \in \mathcal K^n_o$, $g_K: \mathbb R^n \to [0,\infty]$ is defined as $g_K(x) = \min \{r: rx \in K\}$. It is also $1$-homogeneous. The polar body of a convex body $K \in \mathcal K^n_o$ is defined as $K^\circ = \{x \in \mathbb R^n: \langle x, y\rangle \le 1\,\forall y \in K\}$, so that $K^\circ$ is a convex body, and we have $g_K = h_{K^\circ}$; in particular, $g_K$ is convex. The radial function $r_K$ of $K$ is given by $r_K = g_K^{-1}$, which is positively homogeneous of degree $-1$.

Let $\mathcal H^{n - 1}$ denote the $(n - 1)$-dimensional Hausdorff measure on $\mathbb R^n$. For $\mathcal H^{n - 1}$-almost every $x \in \partial K$, there exists a unique normal vector to $K$ at $x$, namely, $u \in S^{n - 1}$ such that $h_K(u) = \langle x, u\rangle$; denote this vector by $n_K(x)$.

For $u \in S^{n - 1}$, $h_K$ is differentiable at $u$ if and only if there exists a unique $x \in \partial K$ such that $n_K(x) = u$, and in this case we have $\nabla h_K(u) = x$. This condition holds for almost every $u \in \partial S^{n - 1}$.

The space of $n\times n$ matrices, $\mathcal M^{n\times n}$, has a natural inner product structure given by $\langle A, B\rangle = \sum_{ij} A_{ij} B_{ij}$, called the Hilbert-Schmidt inner product.
For vectors $x, y \in \mathbb R^n$, we denote by $x  \otimes y$ the matrix in $\mathcal M^{n \times n}$ defined by the linear transformation $z \mapsto \langle y, z\rangle x$. We have $(x \otimes y)^T = y \otimes x$, $\Tr(x \otimes y) = \langle y, x\rangle$, and more generally, $\langle M, x \otimes y\rangle = \langle y, Mx\rangle$ for any $M \in \mathcal M^{n \times n}$.
We use $\Sym^{n\times n}$ to denote the subspace of $\mathcal M^{n\times n}$ of symmetric matrices with the induced inner product, and write $A_{sym}$ for the symmetric part of a matrix $A$, defined by $A_{sym} = \frac{A + A^T}{2}$. In particular, we will frequently use $(x  \otimes y)_{sym} = \frac{1}{2} (x \otimes y + y \otimes x)$.

We denote the space of symmetric positive-definite matrices by $\mathcal P^n$, an open convex cone in $\Sym^{n\times n}$ whose boundary consists of the symmetric positive-semidefinite matrices with nonzero kernel. In the remainder of the paper, ``positive-definite'' is always shorthand for \textit{symmetric} positive-definite. We will use $P^{\frac{1}{2}}$ to denote the unique positive square root of a matrix in $\mathcal P^n$.

We write $GL_n$ for the (real) general linear group in $n$ dimensions, an open submanifold of $\mathcal M^{n \times n}$, $SL_n$ for the subgroup consisting of matrices of determinant $1$, and $O_n$ for the orthogonal group, a compact submanifold of $GL_n$. We will use the fact that the Lie algebras of $SL_n$ and $O_n$ are
\begin{align}
\mathfrak{sl}_n(\mathbb R) &= \{A \in \mathcal M^{n \times n}: \Tr\,A = 0\}, \\
\mathfrak{o}_n(\mathbb R) &= \{A \in \mathcal M^{n \times n}: A^T = -A\}.
\end{align}
Concretely, this just means that if $A$ has zero trace, $e^A \in SL_n$ and if $A$ is antisymmetric then $e^A \in O_n$; moreover, the tangent spaces to $SL_n$ and $O_n$ at $I$ (considered as submanifolds of $GL_n$) are precisely $\mathfrak{sl}_n(\mathbb R)$, $\mathfrak{o}_n(\mathbb R)$, respectively.

Given any positive-definite quadratic form $Q$ (which we identify with the corresponding matrix in $\mathcal P^n$), we use $O(Q)$ to denote the subgroup of $GL_n$ preserving $Q$, namely $Q^{-\frac{1}{2}} O_n Q^{\frac{1}{2}}$. Every conjugate to $O_n$ in $GL_n$ is of this form.

We often denote an affine transformation $x \mapsto Ax + z$ by $(A, z)$; if $A \in \mathcal P^n$, we call this a positive affine transformation, and call the image of a convex body $K$ under such a transformation a positive image of $K$.

Denote by $B^n_p$ the unit ball of the $\ell_p$-norm in dimension $n$; in particular, $B^n_2$ is the Euclidean ball. The Hausdorff metric on $\mathcal K^n$ is defined via 
\begin{equation}
\delta(K, L) = \inf \{d > 0: \text{$L \subset K + d B^n_2$ and $K \subset L + dB^n_2$}\}.
\end{equation}
Many natural functionals such as volume, as well as operations such as translation, multiplication by a matrix, intersection, etc., are continuous with respect to the Hausdorff metric. In the sequel, when we speak of topological properties of $\mathcal K^n$, we always refer to the topology induced by the Hausdorff metric.

A convex body $K \in \mathcal K^n$ is said to be a polytope if it is the convex hull of a finite number of points. $K$ is said to be $C^k$ if $h_K$ is $C^k$ on $\mathbb R^n \backslash\{0\}$, and is said to be $C^k_+$ if $K$ is $C^k$ for $k \ge 2$ and the Hessian of $\left.h_K\right|_{S^{n - 1}}$ is positive-definite at each point of $S^{n - 1}$. The set of polytopes and the set of $C^k_+$ bodies are both dense in $\mathcal K^n$.



\subsection{Some linear-algebraic facts}
We collect here some lesser-known facts from linear algebra, chiefly related to the interplay between positive-definite matrices and orthogonal matrices, which will be used pervasively in the sequel. (The reason for placing these facts in their own section is that the proofs are all based on similar ideas.) We claim no originality for these results.

\begin{lemma}\label{sym_group} Let $S \subset \mathbb R^n$ be a compact set with nonempty interior. Then the group of affine automorphisms $\Aut(S) = \{(A, z) \in \mathcal M^{n \times n} \times \mathbb R^n\,|\, A S + z = S\}$ is compact in $SL_n \ltimes \mathbb R^n$.
\begin{proof}First, since $S$ has nonempty interior, we see that any $\varphi = (A, z)$ must preserve volume, and in particular lie in $SL_n$. Since $(A, z) \in \Aut(S)$ if and only if $As + z \in S$ and $A^{-1}(s - z) \in S$ for all $s \in S$, $\Aut(S)$ is closed. Let $C = S - S$, a compact set containing the origin in its interior; we have $\{(A, 0): (A, z) \in \Aut(S)\} \subset \Aut(C)$. Letting $r, R$ such that $r B^n_2 \subset C \subset R B^n_2$, we see that every coordinate of every $A \in \Aut(C)$ is bounded by $\frac{R}{r}$. Hence the coordinates of any $A$ such that $(A, z) \in \Aut(S)$ are bounded, which implies that $\{z: (A, z) \in \Aut(S)\}$ is also bounded. This implies that $\Aut(S)$ is compact.
\end{proof}
\end{lemma}

\begin{corollary}With the above notation, there exists a (not necessarily unique) quadratic form $Q \in \mathcal P^n$ such that $\Aut(S) \subset \{(A, b(S) - A b(S)): A \in O(Q)\}$, where $b(S) = \frac{1}{|S|}\int_S x\,dx$ is the barycenter of $S$.
\begin{proof}Any symmetry of $S$ must preserve its barycenter, so $(A, z) \in \Aut(S)$ implies $b(S) = b(AS + z) = A b(S) + z$. In addition, the projection of the compact group $\Aut(S)$ on the $GL_n$ factor is a compact subgroup of $GL_n$, and hence is contained in some maximal compact subgroup. The maximal compact subgroups of $GL_n$ are precisely subgroups of the form $O(Q)$.
\end{proof}
\end{corollary}

In particular, since the eigenvalues of any matrix in $O(Q)$, being similar to an orthogonal matrix, are complex numbers of absolute value $1$, the only positive affine transformation $(P, z)$ which can be a symmetry of a compact set with nonempty interior is $(I, 0)$. In fact, we can say a bit more:

\begin{corollary}\label{pos_aff_sym} For any compact set $S$ with nonempty interior and any two distinct positive affine transformations $(P, z), (P', z')$, $PS + z \neq P' S + z'$.
\begin{proof}If $PS + z = P' S + z'$, then $S = P^{-1}(P' S + z' - z)$, so $P^{-1} P'$ lies in some conjugate to the orthogonal group. But $P^{-1} P'$ is similar to the positive-definite matrix $P^{-\frac{1}{2}} P' P^{-\frac{1}{2}}$, and in particular has positive eigenvalues, while the only matrix in the orthogonal group whose eigenvalues are all positive is $I$. Hence $P = P'$, which forces $z = z'$ (by considering the barycenter, say).
\end{proof}
\end{corollary}

In the above proof, we get around the fact that the positive-definite matrices do not form a group by noticing that all we needed to know about a certain matrix is that it is similar to a positive-definite matrix, a property which the product of two positive-definite matrices does have. The same trick is used to prove the following generalization of the polar decomposition:

\begin{lemma}\label{gen_pol_decomp} For any nonsingular matrix $M \in GL_n$, the map $f: \mathcal P^n \times O_n \to GL_n$ defined by $f(P, U) = PMU$ is a global diffeomorphism.
\begin{proof}We will show that $f$ is injective and has nonvanishing derivative, and that there exists a function $g: GL_n \to \mathcal P^n \times O_n$ such that $f \circ g = Id$, which in particular implies that $f$ is surjective. 

For injectiveity, suppose there exist $A \in GL_n$ $P_1, P_2$ positive and $U_1, U_2 \in O_n$ such that $A = P_1 M U_1 = P_2 M U_2$. Then $AA^T = P_1 M M^T P_1 = P_2 M M^T P_2$, so $P_1^{-1} P_2$ preserves the positive-definite quadratic form $M M^T$. In particular, $P_1^{-1} P_2$ has a full set of eigenvalues which are complex numbers of absolute value $1$; but $P_1^{-1} P_2$ is also similar to the positive-definite symmetric matrix $P_1^{-\frac{1}{2}} P_2 P_1^{\frac{1}{2}}$, so its eigenvalues are real and positive; hence we must have $P_1^{-1} P_2 = I$, i.e., $P_1 = P_2$. It follows immediately that $U_1 = U_2$ as well, so $f$ is injective, as desired.

Next, given a point $(P, U) \in \mathcal P^n \times O_n$, consider any nonzero tangent vector in $T_{(P, U)} (\mathcal P^n \times O_n)$ represented by a path $\alpha(t) = (P + tS, U e^{tN})$ to $(P, U)$, with $S$ symmetric and $N$ antisymmetric. We have $(f \circ \alpha)'(0) = \left.\frac{d}{dt}\right|_{t = 0} (P + tS) M U e^{tN} = SMU + PMUN$, and if this vanishes then $P^{-1} S = -(MU) N (MU)^{-1}$. As $N$ is antisymmetric, its eigenvalues are pure imaginary, while $P^{-1} S$ is conjugate to a symmetric matrix and hence has real eigenvalues. Thus $S = N = 0$, contradiction. Hence $f$ is a local diffeomorphism.

Finally, the function $g$ inverting $f$ is given by the explicit solution
\begin{equation}\label{gen_sq_root}
P = Y^{-\frac{1}{2}} (Y^{\frac{1}{2}} AA^T Y^{\frac{1}{2}})^{\frac{1}{2}} Y^{-\frac{1}{2}}
\end{equation}
to the equation $A A^T = P Y P$ we derived above, where $Y = M M^T$; $U$ is then computed as $(PM)^{-1} A$. Clearly, $g: A \mapsto (P, U)$ satisfies $f \circ g = Id$, so we are done.
\end{proof}
\end{lemma}

\begin{remark}\mbox{}
\begin{enumerate}
  \item It can be shown easily by elementary matrix manipulations that this lemma is equivalent to a ``twisted polar decomposition'': if $Q$ is any quadratic form, then every $A \in GL_n$ may be written uniquely as $PV$ with $P \in \mathcal P^n$, $V \in O(Q)$. Unlike the usual polar decomposition, this is \textit{not} a Cartan decomposition of $GL_n$. (For the definition and properties of Cartan decompositions of Lie groups, see \cite{Kn}.) Using the results of Mostow on self-adjoint forms of Lie groups \cite{M}, one may similarly show the existence of ``twisted Cartan decompositions'' of a (real or complex) Lie group $G$, in which the Cartan pair $(\mathfrak l, \mathfrak p)$ associated to a given Cartan involution on the Lie algebra $\mathfrak g$ is replaced by $(g^{-1}\mathfrak l g, \mathfrak p)$ for any $g \in G$. This is, however, beyond the scope of the present paper.

  \item The most general version of the decomposition is the following: for any fixed $M \in GL_n$, any conjugate $O(Q)$ of the orthogonal group, and any conjugate $\mathcal P_R = R^{-1} \mathcal P^n R$ to the cone $\mathcal P^n$, the map $f: \mathcal P_R \times O(Q) \to GL_n$ defined as $f(P, V) = PMV$ is a diffeomorphism. This can be proven by generalizing the proof of the lemma, or by reducing it to the lemma via slightly tedious matrix manipulations.
\end{enumerate}
\end{remark}

\section{Positive John position}\label{section:pos_john}
As with the ordinary John position, the fundamental fact about the maximal-volume image of a convex body $L$ under positive affine transformations contained in a body $K$ is that it is unique. For completeness, we provide a proof of this fact in analogy to the folklore proof of the uniqueness of the maximal-volume ellipsoid (see, e.g., \cite[Proposition 2.1.6]{AAGM}).

\begin{proposition}\label{pos_john_ex} Let $K, L \in \mathcal K^n_o$ be convex bodies, and consider the set of positive images of $L$ inside $K$,
\begin{equation}
\mathcal A_{K, L} = \{PL + z: P \in \mathcal P^n, z \in \mathbb R^n\,|\,PL + z\subset K\}.
\end{equation}
Then there is a unique element in $\mathcal A_{K, L}$ of maximal volume.
\begin{proof}Let $r$ such that $r B^n_2 \subset L$ and $R$ such that $K \subset R B^n_2$. Then for any $(P, z)$ such that $PL + z \subset K$ we must have that the operator norm of $P$ satisfies $\|P\| \le \frac{2R}{r}$, $|z| \le R$ and hence a maximizer of $\{\det P: PL + z \subset K\}$ exists by compactness.

As for uniqueness, assume for the sake of contradiction that $L_i = P_i L + x_i$, $i = 1, 2$ both have maximal volume among positive images of $L$ contained in $K$. Consider two cases: if $P_1 \neq P_2$ then by the strict log-concavity of the determinant on $\mathcal P_n$ \cite[Lemma B.4.1]{AAGM}, $\det(\frac{P_1 + P_2}{2}) > \det(P_1) = \det(P_2)$; but $\frac{P_1 +  P_2}{2} L + \frac{x_1 + x_2}{2} \subset K$ by convexity, contradicting the assumption that $L_1, L_2$ have maximal volume. Otherwise, we have $P_1 L + y, P_1 L + x \subset K$ for some $x \neq y$, and we may assume $y = 0$, $x = e_1$; in particular we have $P_1 L + [0, 1] e_1 \subset K$. Let $[a, b] e_1$ be the projection of $L$ on the $e_1$-axis, let $\epsilon$ such that $\epsilon [a,b] \subset [-\frac{1}{2}, \frac{1}{2}]$, and let $P = P_1 + \epsilon e_1 \otimes e_1$. Then
\begin{equation}
P L + \frac{e_1}{2} \subset P_1 L + \epsilon([a, b] e_1) + \frac{e_1}{2} \subset P_1 L + \left(\frac{1}{2} + \left[-\frac{1}{2},\frac{1}{2}\right]\right) e_1 = P_1 L + [0, 1]e_1 \subset K.
\end{equation}
But one has $\det(P) = \det(P_1 + \epsilon e_1 \otimes e_1) > \det(P_1)$ (as one sees, e.g., by simultaneously diagonalizing $P_1$ and $e_1 \otimes e_1$) and hence $\vol(PL + \frac{e_1}{2}) > \vol(P_1 L)$, contradicting the assumption that $P_1 L$ has maximal volume. Hence the maximal volume element of $\mathcal A_{K, L}$ is unique.
\end{proof}
\end{proposition}

As the term ``position of maximal volume among positive images'' is a bit unwieldy, we call the image of $L$ guaranteed by the above proposition the \textit{positive John image} of $L$ in $K$, and say that a body $K$ is in \textit{positive John position} with respect to $L$, or alternatively that $L$ is in positive John position inside $K$, if the positive John image of $L$ in $K$ is $L$.

We note that not only is it the case that the positive John image $L' = P L + z$ inside $K$ is unique, but the pair $(P, z)$ such that $L' = PL + z$ is uniquely determined as well, even if $L$ has a nontrivial symmetry group. This follows immediately from Corollary \ref{pos_aff_sym}.

In the case of the usual position of maximal volume, it obviously holds that $L' = AL + z$ is the affine image of $L$ of maximal volume contained in $K$ if and only if $L'$ is itself in maximal volume position in $K$. For positive John position a subtlety arises due to the fact that the positive-definite matrices do not form a group, so that for a given image $L' = PL + z$, the family of positive images of $L'$, $\{Q(PL + z) + w: Q \in \mathcal P^n, w \in \mathbb R^n\}$, does not coincide with the family of positive images of $L$. Hence it is not necessarily the case that if $L'$ is the positive John image of $L$ in $K$ then $L'$ is itself in positive John position in $K$. Instead, we have the following characterization:

\begin{proposition}\label{prop:pos_john_im} For any two convex bodies $K, L$ with non-empty interior and a positive definite $P$, the body $L' = PL + z$ is the positive John image of $L$ in $K$ if and only if $P^{\frac{1}{2}} L + P^{-\frac{1}{2}} z$ is in positive John position inside $P^{-\frac{1}{2}} K$.
\end{proposition}
 
\begin{proof}The map $Q \mapsto P^{\frac{1}{2}} Q P^{\frac{1}{2}}$ is a bijection from $\mathcal P^n$ to itself and $\frac{\det(P^{\frac{1}{2}} Q P^{\frac{1}{2}})}{\det(Q)} = \det(P)$ is independent of $Q$, so
\begin{equation}
\begin{array}{lrl}
                 &(P, z) &= \argmax \{\det(Q): (Q, z') \in \mathcal P^n \times \mathbb R^n \,|\, QL + z' \subset K\} \nonumber \\
\Leftrightarrow  &(I, z) &= \argmax \{\det(Q): (Q, z') \in \mathcal P^n \times \mathbb R^n \,|\, P^{\frac{1}{2}} Q P^{\frac{1}{2}} L + z' \subset K\} \nonumber \\
\Leftrightarrow &(I, z) &= \argmax \{\det(Q): (Q, z') \in \mathcal P^n \times \mathbb R^n \,|\, Q (P^{\frac{1}{2}} L) + P^{-\frac{1}{2}} z' \subset P^{-\frac{1}{2}} K\} \nonumber \\
\Leftrightarrow &\quad (I, 0) &= \argmax \{\det(Q): (Q, z') \in \mathcal P^n \times \mathbb R^n \,|\, Q (P^{\frac{1}{2}} L + P^{-\frac{1}{2}} z) + z' \subset P^{-\frac{1}{2}} K\}
\end{array}
\end{equation}
which precisely means that $P^{\frac{1}{2}} L + P^{-\frac{1}{2}} z$ is in positive John position inside $P^{-\frac{1}{2}} K$.
\end{proof}

Following \cite{BR,GLMP}, given two convex bodies $K, L \in \mathcal K^n_o$ we say that $(x, y)$ is a \textit{contact pair} of $K, L$ if $x \in \partial L \cap \partial K$, $y \in \partial L^\circ \cap \partial K^\circ$, and $\langle x, y\rangle = 1$. In other words, $x$ is a common boundary point of $K, L$ and $y$ defines a supporting hyperplane to $K$ and $L$ at $x$.

The analogue of John's theorem in this setting, already quoted in the introduction, characterizes the positive John position of $K$ with respect to $L$ in terms of contact pairs. As we stated, the theorem was first given in \cite[Theorem 4]{BR}, and reproven by different methods as \cite[Corollary 4.4]{GLMP}, but the statement in the latter paper contains a small error. To make the paper self-contained, we give a proof here, half of which follows \cite{GLMP} and half of which follows \cite{BR}.

\begin{manualtheorem}{\ref{pos_john_contact}} 
	Let $K \in \mathcal K^n_o$, $L \in \mathcal K^n$. Then $K$ is in positive John position with respect to $L$ if and only if $L \subset K$ and there exist contact pairs $(x_1, y_1), \ldots, (x_m, y_m)$  of $K, L$ and $c_1, \ldots, c_m > 0$ such that:
	\begin{align}
	0 &= \sum_{i = 1}^m c_i y_i \label{iso_trans} \\
	I_n &= \sum_{i = 1}^m c_i (x_i \otimes y_i)_{sym} \label{iso_dil} 
	\end{align}
\end{manualtheorem}  

\begin{remark}\mbox{}
A useful way of restating the conclusion of the theorem is that $(I_n, 0) \in \Sym^{n\times n}(\mathbb R) \times \mathbb R^n$ lies in the positive convex cone spanned by the set
\begin{equation} C_{K, L} = \{((x \otimes y)_{sym}, y): \text{$(x, y)$ a contact pair of $K$ and $L$}\}.
\end{equation}
In the course of the proof we shall see a more precise characterization.
\end{remark}

As in John's original proof, the necessity part of the theorem follows easily from the following result, which is an extension of the method of Lagrange multipliers to the case where the number of constraints may be infinite:

\begin{theorem}[John \cite{J}]\label{john_opt} Let $V$ be a real vector space of dimension $N$ and $U$ an open neighborhood in $V$, $F: U \to \mathbb R$ a $C^1$ function, $S$ a compact metric space and $G: U \times S \to \mathbb R$ a continuous function such that $\nabla_u G(u, s)$ exists for every $u \in U$, $s \in S$ and $\nabla_u G$ is continuous on $S$. (In optimization terms, $F$ is the objective function and $G$ represents the set of constraints.)

Let $A = \{u \in U\,|\, G(u, s) \le 0\, \forall s \in S\}$ (the feasible set) and $u_0 \in A$ such that $F(u_0) = \max_{u \in A} F(u)$. Then either $\nabla_u F(u_0) = 0$, or there exist $s_1, \ldots, s_m \in S$, $m \le N$ and $\lambda_1,\ldots,\lambda_m \in \mathbb R^+$ such that $G(u_0, s_i) = 0$ for all $i$ and
\begin{equation}
\nabla_u F(u_0) = \sum_{i = 1}^m \lambda_i \nabla_u G(u_0, s_i).
\end{equation}
\end{theorem}

\begin{proof}[Proof of Theorem \ref{pos_john_contact}] For necessity, we apply Theorem \ref{john_opt} with $V = \Sym^{n\times n} \times \mathbb R^n$ and $U = \mathcal P^n \times \mathbb R^n \subset V$. The objective function $F: U \to \mathbb R$ will be defined by $F(P, z) = \det P$, and the constraint function by $G: U \times \partial L \times \partial K^\circ \to \mathbb R$, $G(P, z, x, y) = \langle y, Px + z\rangle - 1$. Since $a \in \mathbb R^n$ lies in $K$ if and only if $\langle y, a\rangle \le 1$ for all $y \in \partial K^\circ$, then $PL + z \subset K$ if and only if $P(\partial L) + z \subset K$, which happens if and only if $G(P, z, x, y) \le 0$ for all $(x, y) \in \partial L \times \partial K^\circ$.

Thus, since $L$ is in positive John position, $\max \{F(P): G(P, z, x, y) \le 0 \,\forall x, y\}$, is attained at $(P, z) = (I, 0)$. The gradient of $F$ is given by $\nabla_{P, z} F(I, 0) = (I, 0)$ (recall that we give $\Sym^{n\times n}$ the Euclidean structure induced by the Hilbert-Schmidt inner product), which is non-vanishing, so there must exist $(x_i, y_i) \in L \times K^\circ$, $i = 1, \ldots, m$ such that
\begin{equation}
(I, 0) = \sum_{i = 1}^m \lambda_i \nabla_{P, z} G(I, 0, x_i, y_i).
\end{equation}
and $G(I, 0, x_i, y_i) = 1$.

On $\mathcal M^{n\times n}$, the gradient of $G(\cdot, z_0, x_i, y_i)$ is $x_i \otimes y_i$ (since $\langle P, x_i \otimes y_i\rangle = \langle y, Px\rangle$); as we are working in the ambient space which is $\Sym^{n\times n}$, the gradient of $G$ in the first variable is the symmetric part of this matrix, $(x \otimes y)_{sym}$. Hence we have for some $\lambda_i \ge 0$ that
\begin{equation}
(I, 0) = \sum_{i = 1}^m \lambda_i ((x_i \otimes y_i)_{sym}, y_i).
\end{equation}

It remains only to show that the $(x_i, y_i)$ are contact pairs. First of all,  $G(I, 0, x_i, y_i)  = 0 $ and so $\langle y_i, x_i\rangle = 1$. By assumption $x_i \in \partial L$ and $y_i \in \partial K^\circ$. As $L \subset K$, we have $K^\circ \subset L^\circ$, and so $y_i \in L^\circ$, and as there exists $x_i \in L$ such that $\langle y_i, x_i\rangle = 1$, $y_i \in \partial L^\circ$; similarly, as $x_i \in K$ and $\langle y_i, x_i\rangle = 1$ with $y_i \in K^\circ$, we have $x_i \in \partial K$. Hence the $(x_i, y_i)$ are contact pairs, and the necessity part of the theorem is proved.

For sufficiency, we restate the conditions \eqref{iso_trans}, \eqref{iso_dil}. Define
\begin{equation}\label{ckl_def}
C_{K, L} = \{((x \otimes y)_{sym}, y): \text{$(x, y)$ is a contact pair of $K$ and $L$}\} \subset \Sym^{n\times n} \times \mathbb R^n.
\end{equation}
Note that if $\sum c_i (x \otimes y)_{sym} = I_n$ then taking traces shows that $\sum c_i = n$. Hence, we wish to show that if $L \subset K$ but $K$ is not in positive John position with respect to $L$, then $(\frac{I_n}{n}, 0) \not\in \conv(C_{K, L})$. If $C_{K, L}$ is empty, we are done, so assume $\partial K \cap \partial L$ is nonempty. As $K$ is not in positive John position with respect to $L$, there exist $P \in \mathcal P^n$, $z \in \mathbb R^n$ such that $L' = P L + z \subset K$ and $\vol(L') > \vol(L)$, i.e., $\det(P) > 1$. In particular, the bodies $L_\lambda = ((1 - \lambda) I + \lambda P)L + \lambda z \subset (1 - \lambda) L + \lambda L'$ are contained in $K$.

Next, note that for any $x \in \partial K \cap \partial L$, any supporting hyperplane of $K$ passing through $x$ must be a supporting hyperplane of $L$ as well. Indeed, if $\langle y, x\rangle = h_K(y)$ then $h_K(y) = \langle y, x\rangle \le h_L(u) \le h_K(u)$, where the last inequality holds because $L \subset K$.

Since $(I + \lambda(P - I)) L + \lambda z \subset K$, we obtain that for any $x \in \partial L \cap \partial K$, $(1 - \lambda) x + \lambda (Px + z) \in K$, so for any $y$ supporting $L$ and $K$ at $x$,
\begin{equation}
\varphi(\lambda) = \langle y, x + \lambda (P - I)x  + \lambda z\rangle \le 1
\end{equation}
for all $\lambda \in [0, 1]$. Since $\varphi(0) = 1$, we must have $\varphi'(0) \le 0$, which yields
\begin{equation}
\langle y, (P - I) x\rangle + \langle y, z\rangle \le 0
\end{equation}
for any contact pair $(y, x)$ of $K, L$. Define the linear functional $\psi$ on $\Sym^{n\times n}(\mathbb R) \times \mathbb R^n$ as
\begin{equation}\psi(M, z) = \langle M, P\rangle_{HS} + \langle y, z\rangle,
\end{equation}
where $\langle M, P\rangle_{HS} = \Tr(M^T P) = \sum_{ij} M_{ij} P_{ij}$ is the Hilbert-Schmidt inner product. Recalling the definition of $C_{K, L}$ from Equation \eqref{ckl_def}, we see that $\psi(\alpha) \le 1$ for all $\alpha = ((x \otimes y)_{sym}, y) \in C_{K, L}$. It is thus sufficient to show that $\psi(\frac{I_n}{n}, 0) > 1$, which would mean that there is a separating hyperplane between $C_{K, L}$ and $(\frac{I_n}{n}, 0)$, and in particular, $(\frac{I_n}{n}, 0)$ does not belong to $\conv(C_{K, L})$.

We have $\psi\left(\frac{I_n}{n}, 0\right) = \frac{1}{n}\Tr P$. Let $\lambda_1,\ldots,\lambda_n > 0$ be the eigenvalues of $P$; by the inequality of arithmetic and geometric means, we have that
\begin{equation}\label{tr_det_am_gm}
\frac{1}{n}\Tr P = \frac{1}{n}\sum_{i = 1}^n \lambda_i \ge \left(\prod_{i = 1}^n \lambda_i\right)^{\frac{1}{n}} = \left(\det(P)\right)^{\frac{1}{n}} > 1.
\end{equation}
Hence $\psi(\frac{I_n}{n}, 0) > \psi((x \otimes y)_{sym}, y)$ for all $(x, y) \in C_{K, L}$. Finally, note that $C_{K, L}$ is the image of the compact set 
$$((\partial K \cap \partial L) \times (\partial K^\circ \cap \partial L^\circ)) \cap \{(x, y) \in \mathbb R^n \times \mathbb R^n | \langle x, y\rangle = 1\}$$
 under a continuous map, so $C_{K, L}$ is compact. Thus implies a hyperplane separation between $(\frac{I_n}{n}, 0)$ and $C_{K, L}$, as desired. The proof is complete.
\end{proof}

\begin{remark}Gordon, Litvak, Meyer, and Pajor showed that by translating a maximal-volume pair $L \subset K$, one can assume that not only the contact points in $\partial K^\circ \cap \partial L^\circ$ are ``centered'' when weighted appropriately, but also the corresponding contact points in $\partial K \cap \partial L$. More precisely, suppose $L$ is in a position of maximal volume in $K$, with contact pairs $(x_i, y_i)_{i = 1}^m$ and weights $c_i$ decomposing the identity as in Theorem \ref{gen_john_thm}. Then setting 
\begin{align}\label{contact_rescale}
z &= \frac{1}{n + 1}\sum_{i = 1}^m c_i x_i \in \frac{n}{n + 1} L \nonumber \\
u_i &= x_i - z \nonumber \\
\gamma_i &= (1 - \langle y_i, z\rangle)^{-1} \\
v_i &= \gamma_i y_i \nonumber \\
a_i &= \gamma_i^{-1} c_i, \nonumber
\end{align} 
one sees that $L - z$ is in a position of maximal volume in $K - z$, that $(u_i, v_i)_{i = 1}^m$ are contact pairs of $(L - z, K - z)$, and that the weights $a_i$ and contact pairs $(u_i, v_i)$ satisfy
\begin{align}
0 &= \sum_{i = 1}^m a_i u_i = \sum_{i = 1}^m a_i v_i, \\
I_n &= \sum_{i = 1}^m a_i (u_i \otimes v_i).
\end{align}
The same exact proof goes through in our situation: if $L$ is in positive John position in $K$, by translating $K$ and $L$ and transforming the contact pairs and weights according to the formulas \eqref{contact_rescale}, we get a positive John position of $L$ in $K$ with contact pairs $(u_i, v_i)$ and weights $a_i$ which, besides satisfying the conclusions of Theorem \ref{pos_john_contact}, also satisfy $\sum a_i u_i = 0$. For details, see \cite[Theorem 3.8]{GLMP} and its proof.
\end{remark}

The existence and uniqueness of the positive John image of $L$ in $K$ makes it a natural position to consider; however, it obviously doesn't tell us anything about images of $L$ under matrices which don't lie in $\mathcal P^n$. By the polar decomposition, we can write the set of affine images of a convex body $L$ as $\{P(UL) + z: P \in \mathcal P^n, U \in O_n, z \in \mathbb R^n\}$, and for each $U \in O_n$ there exist unique $P^*(U), z^*(U)$ such that $P^*(U) UL + z^*(U)$ is the positive John image of $UL$ in $K$. The function $(P^*, z^*): O_n \to \mathcal P^n \times \mathbb R^n$ thus encapsulates all the ``interesting'' information about the affine images of $L$ contained in $K$. We call the family of bodies $P^*(U) (UL) + z^*(U)$ the \textit{positive John family} of $L$ in $K$; we will occasionally also abuse terminology and use the term positive John family to refer to the function $(P^*, z^*)$.

Note that $\max_{U \in O_n} \det(P^*(U))$ picks out the position of maximal volume of $L$ in $K$ considered in \cite{GLMP}, but we shall see that other properties of $(P^*, z^*)$ are of interest as well.

As the function $(P^*, z^*)$ is a solution to a parametrized optimization problem, the most natural toolbox with which to investigate its properties is the toolbox of mathematical economics, which uses the language of set-valued analysis. We will introduce some basic concepts of set-valued analysis below and use them freely in proving our results, though the proofs may be reformulated to avoid their use. Later, in Section \ref{section:max_int_pos} we will encounter results whose statement, and not just proof, requires concepts from set-valued analysis, so that the use of set-valued analysis cannot be avoided in any case. Our main source for the material below is \cite[Chapter 17]{AB}.

\subsection{Set-valued analysis}\label{subsection:set_valued}
Let $X, Y$ be topological spaces. The fundamental objects of study in set-valued analysis, unsurprisingly, are set-valued functions $f: X \to P(Y)$. Such functions are called correspondences and written $f: X \rightrightarrows Y$. We say that a correspondence is open-valued, closed-valued, compact-valued, convex-valued, etc., if $f(x)$ is open, closed, compact, convex, etc. for each $x \in X$.

\begin{definition} A correspondence $f: X \rightrightarrows Y$ is called \textit{upper hemicontinuous} at $a \in X$ if for any open set $V$ containing $f(a)$ there exists a neighborhood $U$ of $a$ such that for all $x \in U$, $f(x) \subset V$.

Conversely, $f$ is called \textit{lower hemicontinuous} at $a$ if for any open set $V$ intersecting $f(a)$ there exists a neighborhood $U$ of $a$ such that $f(x)$ intersects $V$ for all $x$ in $U$.
\end{definition}

It is immediate that for an ordinary function, considered as a set-valued function, both upper or lower hemicontinuity are equivalent to continuity.

It will be useful to cite an equivalent characterization of hemicontinuity in terms of sequences:

\begin{proposition}Let $X, Y$ be metric spaces, $f: X \rightrightarrows Y$ a correspondence. If $f$ is compact-valued, then $f$ is upper hemicontinuous iff for all $a_n \to a$, $b_n \in f(a_n)$ such that $b_n \to b$ we have $b \in f(a)$.

Conversely, $f$ is lower hemicontinuous at $a$ iff for all $a_n \to a$, and $b \in f(a)$, there exist a subsequence $\{a_{n_k}\}$ of $\{a_n\}$ and $b_k \in f(a_{n_k})$ such that $b_k \to b$.
\end{proposition}

We will later have use for some natural operations on correspondences. Given two correspondences $f, g: X \rightrightarrows Y$, the intersection correspondence is defined in the obvious way as $(f \cap g)(x) = f(x) \cap g(x)$. Similarly, for a correspondence $f$, the convex hull correspondence $\conv f$ is defined as $(\conv f)(x) = \conv f(x)$.

\begin{proposition}Let $f: X \rightrightarrows Y$ be a compact-valued upper hemicontinuous correspondence. Then:
  \begin{enumerate}
    \item If $Y$ is a metric space and $g: X \rightrightarrows Y$ is a closed-valued correspondence, then $f \cap g$ is upper hemicontinuous.
    \item If $Y = \mathbb R^n$, the convex hull correspondence $\conv f$ is upper hemicontinuous.
  \end{enumerate}
\end{proposition}

A fundamental tool in optimization is the Berge maximum theorem, of which we now cite a version sufficient for our purposes \cite[Theorem 17.31]{AB}:

\begin{theorem}[Maximum theorem]
Let $X$ and $\Theta$ be topological spaces, $f: X \times \Theta \to \mathbb R$ be a continuous function on $X \times \Theta$, and $C: \Theta \rightrightarrows X$ be a compact-valued correspondence (the family of feasible sets) such that $C(\theta) \ne \emptyset$ for all $\theta \in \Theta$. Define the value function $f^*: \Theta \to \mathbb R$ by
\begin{equation}
f^*(\theta) = \sup \{f(x, \theta) : x\in C(\theta)\}
\end{equation}
and the set of maximizers $C^*: \Theta \rightrightarrows X$ by
\begin{equation}
C^*(\theta) =
\argmax\{f(x,\theta): x \in C(\theta)\} = \{x \in C(\theta) : f(x, \theta) = f^*(\theta)\}.
\end{equation}

If $C$ is continuous (i.e. both upper and lower hemicontinuous), then $f^*$ is continuous and $C^*$ is upper hemicontinuous with nonempty and compact values.
\end{theorem}

We can now state and prove the first result of the subsection:

\begin{proposition}\label{pos_john_contin_on}
Let $K, L \in \mathcal K^n$, and for each $U \in O_n$, let $P^*(U), z^*(U)$ be defined such that $P^*(U) UL + z^*(U)$ is the positive image of $L$ with maximum volume in $K$. Then $P^*$ and $z^*$ are continuous functions on $O_n$.
\end{proposition}

We will derive this from a more general proposition which is a simple application of the maximum theorem. (We will have use for the more general version later.) For comparison, we also give a direct proof which avoids the use of set-valued analysis.

\begin{proposition}\label{pos_john_contin} The function $(P^*, z^*): \mathcal K^n \times \mathcal K^n \to \mathcal P^n \times \mathbb R^n$ defined such that $P^*(K, L) L + z^*(K, L)$ is the positive John image of $L$ contained in $K$ is continuous with respect to the Hausdorff metric.
\end{proposition}

Proposition \ref{pos_john_contin_on} follows immediately from Proposition \ref{pos_john_contin} upon noticing that the function $f: O_n \to \mathcal K^n$ defined by $f(U) = UL$ is continuous.

\begin{proof}[First proof of Proposition \ref{pos_john_contin}.]
Let $X = \overline{\mathcal P}_n \times \mathbb R^n$, where $\overline{\mathcal P}_n$ is the set of positive semidefinite matrices (a closed convex cone in the space of $n \times n$ matrices), $\Theta = \mathcal K^n \times \mathcal K^n$, $f(P, z, K, L) = \det(P)$, and $C(K, L) = \{(P, z) \in X: PL + z \subset K\}$. It is not hard to check that $C$ is both upper and lower hemicontinuous. For upper hemicontinuity, recall that in the course of the proof of Proposition \ref{pos_john_ex} we saw that $C(K, L)$ is compact, so we can use the sequential characterization: if $(K_m, L_m) \to (K, L)$, $(P_m, z_m) \to (P, z)$, and $P_m L_m + z_m \subset K_m$ then clearly $P L + z \subset K$ because $K$ is closed. For lower hemicontinuity, we use the definition: if $(P, z) \in V \cap C(K, L)$ for some open set $V$, then $PL + z \subset K$ by definition of $C(K, L)$, and $((1 - \epsilon)^2 P, (1 - \epsilon) z) \in V$ for some $\epsilon > 0$ because $V$ is open. Now let $U$ be a neighborhood of $(K, L)$ such that $L' \subset (1 - \epsilon)^{-1} L$ and $(1 - \epsilon) K \subset K'$ for all $(K', L') \in U$, so that
\begin{equation}
(1 - \epsilon)^2 P L' + (1 - \epsilon) z \subset (1 - \epsilon)P L + (1 - \epsilon) z \subset (1 - \epsilon) K \subset K',
\end{equation}
giving $((1 - \epsilon)^2P, (1 - \epsilon)z) \in V \cap C(K', L')$ for all $(K', L') \in U$; in particular, $V \cap C(K', L')$ is nonempty.

Hence, by the maximum theorem, $C^*(K, L)$ is upper hemicontinuous; but by Proposition \ref{pos_john_ex}, $C^*(K, L)$ is single-valued, and a single-valued upper hemicontinuous correspondence is continuous, hence $P^*, z^*$ are continuous functions of the pair $K, L$, as desired.
\end{proof}

\begin{proof}[Second proof of Proposition \ref{pos_john_contin}.]
Write $f^*(K, L) = \det(P^*(K, L))$. We first claim that $f^*: \mathcal K^n \times \mathcal K^n \to \mathbb R^+$ is continuous. Indeed, given $K, L$ and $\epsilon > 0$, let $U$ be a neighborhood of $(K, L)$ such that for all $(K', L') \in U$ one has $(1 + \epsilon)^{-1} K \subset K' \subset (1 + \epsilon) K$ and similarly for $L$.
Then if $PL + z \subset K$ then $(1 + \epsilon)^{-1} P L' + z \subset (1+\epsilon) K'$ and similarly with $K, L$ and $K', L'$ interchanged, implying $f^*(K', L') \in ((1 + \epsilon)^{-2n} f^*(K, L), (1 + \epsilon)^{2n} f^*(K, L))$, so $f^*$ is continuous.

Next, we claim that the graph of $(P^*, z^*): \mathcal K^n \times \mathcal K^n \to \overline{\mathcal P}_n \times \mathbb R^n$ is closed: indeed, if $(K_m, L_m) \to (K, L)$ and $(P^*(K_m, L_m), z^*(K_m, L_m)) \to (P, z)$, we must have
\begin{equation}
f(P) = \lim_{m \to \infty} f(P^*(K_m, L_m)) = \lim_{m \to \infty} f^*(K_m, L_m) = f^*(K, L).
\end{equation}
But we have $P^*(K_m, L_m) L_m + z^*(K_m, L_m) \subset K_m$ and hence, taking the limit, $P L + z \subset K$. Thus $P L + z$ must be the positive image of $L$ of maximum volume in $K$, i.e., $(P, z) = (P^*(K, L), z^*(K, L))$, showing that the graph of $(P^*, z^*)$ is closed.

Finally, restricting to the neighborhood $U$ defined above, the range of $(P^*, z^*)$ can be taken to be compact (as before, letting $r, R$ such that $r B^n_2 \subset L$, $K \subset R B^n_2$, we have that
\begin{equation}\im\left.(P^*, z^*)\right|_U \subset \{P \in \overline{\mathcal P}_n: \|P\| \le \frac{R}{r} (1 + \epsilon)\} \times (1 + \epsilon) K,
\end{equation}
which is compact); hence, by the closed graph theorem of point-set topology, $(P^*, z^*)$ is continuous, as desired.
\end{proof}

In particular, one obtains that the John and L\"owner ellipsoids of a convex body $K$ are continuous in $K$, which was pointed out in \cite[p. 966]{ABT}.

The proof of Proposition \ref{pos_john_contin} makes no use of the fact that we are working with positive matrices specifically, and so a similar result can be stated for general affine images: the position of maximal volume of $L$ in $K$ defines an upper hemicontinuous correspondence $\mathcal K^n \times \mathcal K^n \rightrightarrows \mathcal P^n \times \mathbb R^n$. In general, however, this correspondence will not be single-valued.

\subsection{Differentiability of the positive John position}
Our next result shows that under certain conditions on $K, L$, one can obtain stronger regularity of the function $(P^*, z^*)$:

\begin{theorem}\label{pos_john_sm} Let $K \in \mathcal K^n$ be a $C^k_+$ body ($k \ge 2$) and $L \in \mathcal K^n$ be a polytope. Consider the positive image map $(P^*, z^*): O_n \to \mathcal P^n \times \mathbb R^n$ determined by $K, L$. Let $U_0 \in O_n$ such that the contact pairs $(x_i, y_i)$ and constants $c_i$ satisfying the conclusions of Theorem \ref{pos_john_contact} (Equations \eqref{iso_trans}, \eqref{iso_dil}) are uniquely determined. Then $(P^*, z^*)$ is $C^{k - 1}$ in a neighborhood of $U_0$.
\end{theorem}

The idea of the proof is that under these conditions, $(P^*(U), z^*(U))$ is a solution to a finite-dimensional convex optimization problem, and such solutions can be shown to vary regularly with the parameter, under certain conditions on the constraints, by translating the optimization problem into an implicit function problem. The main tool for accomplishing this translation is the method of Lagrange multipliers, which, when inequalities are involved, is known as the Karush-Kuhn-Tucker theorem \cite[Corollaries 5.2.2, 5.2.3]{C}:

\begin{theorem}[Karush-Kuhn-Tucker]\label{kkt_thm} Let $\Omega \subset \mathbb R^N$ be a domain, $f, g_1,\ldots,g_m: \Omega \to \mathbb R$ twice-differentiable functions, and consider the problem of maximizing $f$ on $\Omega$ subject to the constraints $g_1 \le 0, \ldots, g_m \le 0$. Let $\Gamma = \bigcap_{i = 1}^m \{x \in \Omega: g_i(x) \le 0\}$ denote the feasible set.

If $x^*$ is a local optimum of $f$ on $\Gamma$, then there exists $\mu^* \in (\mathbb R_+)^m$ such that $(x^*, \mu^*)$ is a stationary point of the Lagrangian $\mathcal (x, \mu) = f(x) - \sum \mu_i g_i(x)$, i.e., $\left.\nabla_x \mathcal L(x, \mu)\right|_{(x^*, \mu^*)} = 0$.
\end{theorem}

\begin{proof}[Proof of Theorem \ref{pos_john_sm}] First, fix $U \in O_n$. Let $g_K$ be the gauge function of $K$, which is $C^k_+$ by assumption, and let $x_1, \ldots, x_m$ be the vertices of $L$. Set $\Omega = \mathcal P^n \times \mathbb R^n$, define $f: \Omega \to \mathbb R$ by $f(P, z) = \log \det(P)$, and define $g_1, \ldots, g_m: \Omega \to \mathbb R$ by $g_i(P, z) = g_K(P U x_i + z) - 1$. Then $P(UL) + z \subset K$ if and only if $(P, z) \in \Gamma = \{(P, z): g_1(P, z), \ldots, g_m(P, z) \le 0\}$. By assumption, each $g_i$ is strictly convex and $k$ times differentiable. In addition, $P \mapsto \log \det(P)$ is strictly concave on $\mathcal P^n$ by an inequality of Minkowski \cite[Lemma B.4.1]{AAGM}.

We wish to characterize the solution to the optimization problem $\max_{(P, z) \in \Gamma} f(P)$. By the Karush-Kuhn-Tucker theorem, any local optimum $(P_0, z_0)$ of $f$ on the feasible set is a stationary point of the Lagrangian $\mathcal L(P, z, \mu) = f(P) - \sum \mu_i g_i(P, z)$, i.e., $\nabla_{P, z} f(P_0, z_0) = \sum_i \mu_i \nabla_{P, z} g_i(P_0, z_0)$ and $\mu_i g_i(P_0, z_0) = 0$ for all $i$. For every $(P, z) \in \Gamma$, we have
\begin{equation*}
0 \ge \mu_i g_i(P, z) \ge \mu_i \nabla_{P, z} g_i(P_0, z_0) \cdot (P - P_0, z - z_0) + \mu_i g_i(P_0, z_0)
\end{equation*}
 because $\mu_i g_i$ is convex, implying that
\begin{align*}
f(P, z) &\le f(P_0, z_0) + \nabla_{P, z} f(P_0, z_0) \cdot (P - P_0, z - z_0) \\&= f(P_0, z_0) + \sum_i \mu_i \nabla_{P, z} g_i(P_0, z_0)\cdot (P - P_0, z - z_0) \le f(P_0, z_0)
\end{align*}
for any $(P, z) \in \Gamma$, where the first inequality uses the concavity of $f$. Hence any local optimum of $f$ on $\Gamma$ is a global maximum, i.e., corresponds to a positive John image of $UL$ in $K$, and as we have already seen, this image is unique. Thus we see that the positive John image of $UL$ is characterized by the equation $\left.\nabla\right|_{(P, z, \mu)} \mathcal L(P_0, z_0, \mu_0) = 0$.

We now allow $U$ to vary in $O_n$, and consider $\mathcal L$ to be a function of $U$ as well. By what we have seen so far, the graph of the function $(P^*, z^*): O_n \to \Omega$ can be described as the projection onto the first three coordinates of the set $Z = \{(U, P, z, \mu): \left.\nabla\right|_{(P, z, \mu)} \mathcal L(U, P_0, z_0, \mu_0) = 0\}$. Hence we can study the regularity of $(P^*, z^*)$ by means of the implicit function theorem.

Let $U_0$ be as in the statement of the theorem, abbreviate $(P_0, z_0) = (P^*(U_0), z^*(U_0))$ and let $x_{i_1}, \ldots, x_{i_k}$ be the vertices of $L$ whose image under $(P_0, z_0)$ lie on the boundary of $K$. The $g_{i_1}, \ldots, g_{i_k}$ corresponding to $x_{i_1}, \ldots, x_{i_k}$ are precisely the binding constraints of the optimization problem at $U_0$, i.e., we have $g_{i_j}(P_0, z_0) = 0$ for all $j$. Since $(P^*, z^*)$ is continuous in $U$, we know that in a neighborhood of $U_0$, $x_{i_1}, \ldots, x_{i_k}$ are the only vertices of $L$ whose images possibly lie on the boundary of $K$. Restricting to this neighborhood, we may replace the original optimization problem by the problem defined using only the constraints $g_{i_1}, \ldots, g_{i_k}$. Let $\mu^*$ be the Lagrange multiplier for the new problem at $U = U_0$.

The crucial point is that under our assumption on $U_0$, the constraints $g_{i_1}, \ldots, g_{i_k}$ are regular at $U_0$, i.e., $\nabla_{P, z} g_{i_1}, \ldots, \nabla_{P, z} g_{i_k}$ are linearly independent at $(P_0, z_0)$. Indeed, we compute $\nabla_{P, z} g_{i_j}(P_0 U_0 x_{i_j} + z_0) = \left.\nabla_u g_K(u)\right|_{u = P_0 U_0 x_{i_j} + z_0} \cdot A_{x_{i_j}}$, where $A_x$ denotes the linear map $(P, u) \mapsto Px + z$, and for $u \in \partial K$, $v = \nabla_u g_K(u)$ is precisely the contact pair of $u$, namely, the unique vector $v \in \partial K^\circ$ such that $\langle v, u\rangle = 1$ (see e.g.~\cite[1.39]{S}). Under the identification of $\Sym^{n\times n} \times \mathbb R^n$ with its dual via the respective Euclidean structures, we obtain that $\nabla_{P, z} g_{i_j}(P_0 U_0 x_{i_j} + z_0) = ((u_j \otimes v_j)_{sym}, v_j)$, where $(u_j, v_j)$ is a contact pair of $K$ and $P_0 U_0 L + z_0$ defined by $u_j = P_0 U_0 x_{i_j} + z_0$ and $v_j$ is the unique corresponding point on $\partial K^\circ$. Our assumption on $U_0$ precisely means that the equation $\sum c_j (u_j \otimes v_j)_{sym} = (I, 0)$ has a unique solution, so the $((u_j \otimes v_j)_{sym}, v_j)$ must be linearly independent.

Now let $H: O_n \times \Sym^{n\times n} \times \mathbb R^n \times (\mathbb R^k)_+ \to  \Sym^{n\times n} \mathbb R^n \times \mathbb R^k$ be defined by $H(U, P, z, \mu) = \nabla_{P, z, \mu} \mathcal L(U, P, z, \mu)$. We wish to apply the implicit function theorem at $(U_0, P_0, z_0, \mu_0)$ in order to express the zero set of $H$ as the graph of a function; we thus need to show that the Jacobian $J_{U_0}$ of $H(U_0, \cdot)$ is nonsingular. $J_{u_0}$ is precisely the Hessian of $\mathcal L(U_0, P, z, \mu) = f(U_0, P, z) - \sum_{j = 1}^k \mu_j g_{i_j}(U_0, P, z)$; separating derivatives in $(P, z)$ from derivatives in $\mu$, we write this as the block matrix $J_{U_0} = \left(\begin{smallmatrix} D^2_{P, z} \mathcal L & D_{P, z} G \\ (D_{P, z} G)^T & 0\end{smallmatrix}\right)$, where $G = (g_{i_1}, \ldots, g_{i_k})$. Since $f$ is strictly concave and the $g_i$ are strictly convex, $\mathcal L$ is strictly concave in $(P, z)$ and in particular $D^2_{P, z} \mathcal L$ has full rank; in addition, as the constraints $g_{i_1}, \ldots, g_{i_k}$ are regular, $(D_{P, z} G)^T$ also has full rank. It follows that $J_{U_0}$ has full rank, so by the implicit function theorem, the zero set of $H$ is the graph of a $C^{k - 1}$ function $h: O_n \to \Sym^{n\times n} \times \mathbb R^n \times (\mathbb R^k)_+$ in a neighborhood of $U_0$. As we have seen, the projection to the $(P, z)$ variables of the zero set of $H$ is precisely the positive John image map, so we obtain that $(P^*, z^*): O_n \to \Sym^{n\times n} \times \mathbb R^n$ is $C^{k - 1}$ in a neighborhood of $U_0$, as desired.
\end{proof}

\begin{remark}\mbox{}
\begin{enumerate}
\item As stated, the assumptions of the theorem never apply if $K, L$ are centrally symmetric bodies (in which case the positive John image of $L$ is always obtained at $z = 0$). The reason for this is simply that every contact pair $(x_i, y_i)$ will have a corresponding contact pair $(x_j, y_j) = (-x_i, -y_i)$; as both pairs map to the same matrix $x_i \otimes y_i$, there's always a degree of freedom in choosing the coefficients $c_i$ such that $\sum c_i (x_i \otimes y_i)_{sym} = I_n$. To get around this technicality, one chooses a single representative of each pair of vertices $\{x, -x\}$ of $L$, and requires that the coefficients $c_i$ are unique when restricting to contact pairs involving only those vertices. Since both representatives map to the same constraint in the optimization problem encountered in the proof of the theorem, the proof goes through in this case as well.

\item We sketch an example showing that without the assumption on the uniqueness of the solution to Equations \eqref{iso_trans}, \eqref{iso_dil}, the conclusion of the theorem may fail.

Consider the square $S = B_1^2$ and disk $B = B_2^2$ in the plane. Clearly, for any $U \in O_n$, the positive John image of $US$ inside $B$ is $US$. Let $L = \conv(S, \{\pm v\})$ where $v = (\frac{1}{\sqrt 2}, \frac{1}{\sqrt 2})$, and let $K$ be a centrally symmetric $C^2_+$ body such that $\pm e_1, \pm e_2, \pm v \in \partial K$, $\partial K = \partial B$ in a neighborhood of $\pm e_1, \pm e_2$, and $n_K(v)$ is rotated counterclockwise from $v$ by some nonzero angle $\theta$, which means that small clockwise rotations of $v$ will lie in $K$ but not small counterclockwise rotations. For instance, one can take $K$ to look like the ellipsoid $E = \{(x, y): \frac{3}{4}x^2 + \frac{5}{4} y^2 \le 1\}$ in a neighborhood of $v$. See Figure \ref{kl_fig}.

\begin{figure}[h]
\includegraphics[width=.5\textwidth]{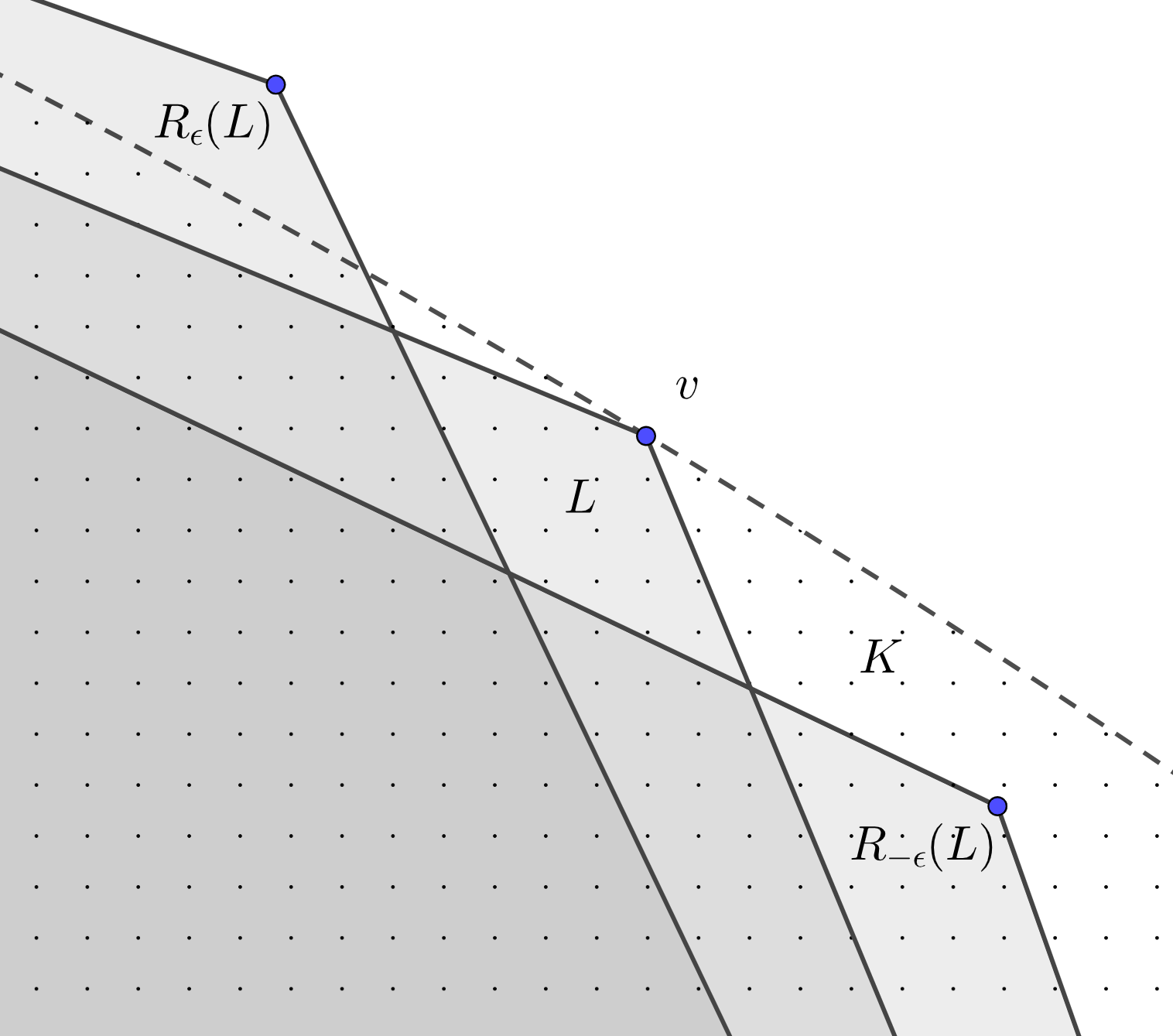}
\caption{$K$, $L$, and small rotations of $L$ in a neighborhood of $v = \left(\frac{\sqrt 2}{2}, \frac{\sqrt 2}{2}\right)$. $L$ and its rotations are shaded; $K$ is dotted. The equation of the dashed line is $\frac{3}{4}x^2 + \frac{5}{4} y^2 = 1$.}
\label{kl_fig}
\end{figure}

We now examine the positive John family of $L$ inside $K$ near the identity. Let $R_\theta = \left(\begin{smallmatrix}\cos\theta & -\sin\theta \\ \sin\theta & \cos\theta\end{smallmatrix}\right)$. Since $R_\theta S$ is in positive John position for all $\theta$ and positive John position is unique, we see that no positive matrix of determinant $1$ satisfies $P R_\theta S \subset B$ except $P = I$; since $K$ looks like $B$ in a neighborhood of $e_1, e_2$ and $L$ contains $S$, we see that for $\theta$ close enough to zero, no positive matrix of determinant $1$ can satisfy $P R_\theta L \subset K$ except possibly $P = I$. For $\theta < 0$, $R_\theta v$ will lie in the interior of $K$, so $R_\theta L$ is in positive John position for such $\theta$, as is the case for $\theta = 0$. However, for $\theta > 0$, $R_\theta v$ will lie outside $K$, and since no positive matrix of determinant $1$ satisfies $P u_1, Pu_2 \in B$ except for the identity, $P^*(R_\theta)$ (which is close to the identity by continuity of positive John position) must have determinant strictly less than $1$. A more careful analysis, which we omit, yields that in fact $\det P^*(R_\theta) = 1 - \Omega(\theta)$ for $\theta > 0$. In particular, $\det P^*$, and hence $P^*$, cannot be differentiable at $I$.

\item A similar theorem may be stated in the case that $K$ is a polytope and $L$ is $C^k_+$. Indeed, the condition $PUL + z \subset K$ can be translated into a finite set of constraints in this case as well: $PUL + z \subset K$ if and only if $h_{PUL + z}(u_i) \le h_K(u_i)$, where $u_i$ are the facet normals of $K$. We can rewrite this:
\begin{align*}h_{PUL + z}(u_i) &= \sup_{x \in PUL + z} \langle x, u_i\rangle = \sup_{y \in UL} \langle y, P u_i\rangle + \langle z, u_i\rangle = h_{UL}(P u_i) + \langle z, u_i\rangle.
\end{align*}
Since $h_{UL}$ is convex, we see that $h_{PUL + z}(u_i)$ is a convex function of $P, z$, and hence $h_{PUL + z}(u_i) \le h_K(u_i)$ is a convex inequality constraint on $(P, z)$. So this case is exactly dual to the case in which $L$ is a polytope, with $h_L$ taking the place of $g_K$.
\item We conjecture that a similar theorem also holds when both $K$ and $L$ are $C^k_+$.
\end{enumerate}
\end{remark}

\section{Saddle-John position}\label{section:saddle_john}
Let $K, L$ be given, and define $(P^*, z^*): O_n \to \mathcal P^n \times \mathbb R^n$ as in the previous section. For $U^* \in \argmax_{O_n} \det P^*(U)$, $L_{max} = P^*(U^*) U^* L + z^*(U^*)$ is a position of maximal volume of $L$ in $K$, regarding which we have finer information than for other positive John images: as shown by \cite{GLMP} (see there for references to earlier work), in this case the contact pairs of $\partial K \cap \partial L_{max}$ support a decomposition of the identity (in the sense of Theorem \ref{gen_john_thm}).
There's another distinguished point on the graph of $\det P^*$, namely the minimum: for given $K, L$, what orthogonal image of $L$ is the ``hardest'' to fit inside $K$, when we allow translations and dilations by positive matrices? Formally, for any $U^* \in \argmin_{O_n} \det P^*(U)$, we consider the image of $L$ given by $L_{sj} = P^*(U^*) U^* L + z^*(U^*)$. Since this image of $L$ is defined by maximizing in one set of variables (the dilation and translation) and minimizing in another (the orthogonal image), we call it the \textit{saddle-John image} of $K$ in $L$; if $L$ is its own saddle-John image, we say that $L$ is in \textit{saddle-John position}.

As with the position of maximal volume, the saddle-John image need not be unique in general: for example, if $L$ is a square and $K$ is a circle, then $(P^*(U), z^*(U)) = (I, 0)$ for all $U \in O_n$ and any rotation of $L$ is simultaneously a position of maximal volume and a saddle-John position.

Saddle points of functions are in particular critical points, which satisfy the same first-order conditions as maxima. Since the optimization result of John (Theorem \ref{john_opt}) which underlies the proof of the properties of the position of maximal volume is obtained by a first-order analysis of the corresponding constrained maximization problem, one might suspect that the saddle-John position of a body $L$ has similar properties as the position of maximal volume, even though it cannot be obtained so simply as the solution to a ``global'' optimization problem defined on all affine images of $L$. Our first goal in this section is to verify this intuition: using Theorem \ref{pos_john_sm} along with an additional tool from mathematical economics, we will show that the saddle-John position, just like the original John position, yields a genuine decomposition of the identity on the contact points of $K$ and $L$. The tool we need is a version of the envelope theorem \cite[Theorem 6.1.1]{C}:

\begin{theorem}\label{envel_thm} Let $X$ and $\Theta$ be smooth manifolds. Consider the parametrized constrained optimization problem
\begin{equation}
\max_{x \in G(\theta)} f(x, \theta)
\end{equation}
with a $C^2$ objective function $f: X \times \Theta \to \mathbb R$ and a parameter-dependent feasible set $G(\theta)$ defined by
\begin{equation}
G(\theta) = \{x \in X: g_j(x, \theta) \le 0, j = 1, \ldots m\},
\end{equation}
where $g_1, \ldots, g_m: X \times \Theta \to \mathbb R$ are $C^2$ functions as well. Define the Lagrangian $L(x, \mu, \theta) = f(x) - \sum_{i = 1}^m \mu_i g_i(x, \theta)$, and the value function $v(\theta) = \sup_{x \in G(\theta)} f(x, \theta)$.

Let $(x_0, \theta_0) \in X \times \Theta$ such that $x_0$ is a strict local maximum of $f(x, \theta_0)$ in $G(\theta_0)$, and suppose that the binding constraints are regular, that is, the vectors $\{\nabla_x g_i(x_0, \theta_0): g_i(x_0, \theta_0) = 0\}$ are linearly independent. Then $v(\theta)$ is $C^1$ in a neighborhood of $\theta_0$, and $\nabla_\theta v(\theta_0) = \nabla_\theta L(x_0, \theta_0, \mu_0)$.
\end{theorem}

\begin{theorem}\label{saddle_john_contact}
Let $K \in \mathcal K^n_o$, $L \in \mathcal K^n$, and suppose $L$ is in saddle-John position inside $K$. Then there exist contact pairs $(x_1, y_1), \ldots, (x_m, y_m)$  of $K, L$ and $c_1,\ldots,c_m > 0$ such that:
\begin{enumerate}
\item $\sum_{i = 1}^m c_i y_i = 0$.
\item $\sum_{i = 1}^m c_i (x_i \otimes y_i) = I_n$.
\end{enumerate}
\end{theorem}
(The first conclusion already follows from the fact that $L$ is in positive John position, but we have included it for completeness in the statement of the theorem.)

\begin{proof}The proof consists of two main steps: first, if $K, L$ are a sufficiently nice pair of bodies and $L$ is in saddle-John position in $K$, the envelope theorem will enable us to use the condition that $\nabla_U \left.\det P^*\right|_{U = I} = 0$ to obtain an decomposition of the identity supported on contact pairs. The existence of a decomposition of the identity for the saddle-John positions of general pairs of bodies will follow by approximation.

First, suppose that $K$ is $C^2_+$, $L = \conv(x_1, \ldots, x_m)$ is a polytope, and that $K, L$ satisfy the hypotheses of Theorem \ref{pos_john_sm}. As in the proof of that theorem, $(P^*(U), z^*(U))$ is the solution to the optimization problem with objective function $f(U, P, z) = \log \det P$ and constraints $g_i(U, P, z) = g_K(P U x_i + z) - 1$; our assumptions imply that $(P^*(I), z^*(I)) = (I, 0)$, that $\nabla_U f(I, I, 0) = 0$, and, as we saw in the course of the proof, that the vectors $\nabla_{P, z} g_i(I, I, 0) = (x_i \otimes y_i, y_i)$ for $x_i \in \partial K \cap \partial L$ are linearly independent, and that the weights $c_i$ in the equation $\sum c_i ((x_i \otimes y_i)_{sym}, y_i) = (I, 0)$ are the Lagrange multipliers for the optimization problem at $U = I$. Letting $U(t) = e^{tA}$ be an arbitrary one-parameter subgroup of $O_n$, with $A$ antisymmetric, the envelope theorem thus yields that
\begin{align*}
0 &= \left.\frac{d}{dt}\right|_{t = 0} f(U(t), P^*(U(t)), z^*(U(t))) = \left.\frac{d}{dt}\right|_{t = 0} \left( f(U(t), I, 0) - \sum_{i: x_i \in \partial K} c_i g_i(U(t), I, 0) \right) \\
&= 0 - \sum_{i: x_i \in \partial K} c_i \left.\frac{d}{dt}\right|_{t = 0} g_K(e^{tA} x_i) = \sum c_i \langle y_i, Ax_i\rangle = \left\langle A, \sum c_i (x_i \otimes y_i)\right\rangle.
\end{align*}
Since $\left\langle A, \sum c_i (x_i \otimes y_i)\right\rangle$ vanishes for all one-parameter subgroups $e^{tA}$, i.e., all antisymmetric matrices $A$, we obtain that $\sum c_i (x_i \otimes y_i)$ is symmetric. Since $L$ is in positive John position, its symmetric part is $I_n$, so we obtain that $\sum c_i (x_i \otimes y_i) = I_n$, as desired.

Now let $K, L$ be arbitrary convex bodies such that $L$ is in saddle-John position inside $K$, and let $(x_i, y_i)$, $i = 1, \ldots, k$ be a minimal set of contact pairs of $K, L$ such that there exist weights $c_i$ satisfying Equations \eqref{iso_trans}, \eqref{iso_dil}. Let $K'$, $L'$ be convex bodies satisfying the following conditions:
\begin{enumerate}
  \item $L' \subset L \subset K \subset K'$.
  \item $K'$ is $C^2_+$ and $L$ is a polytope.
  \item $y_i \in \partial (K')^\circ$ for all $i$ (which implies $x_i \in \partial K'$) and $x_i \in \partial L'$ for all $i$ (which implies $y_i \in (\partial L')^\circ$).
  \item $\partial L \cap \partial K' = \{x_1, \ldots, x_k\}$.
\end{enumerate}
Let $(Q^*, w^*)$ be the positive John families associated to the pair $K', L'$, and $(P^*, z^*)$ the family associated to $K, L$. Since $P^*(U) U L' + z^*(U) \subset P^*(U) U L + z^*(U) \subset K \subset K'$, while $Q^*(U) U L' + w^*(U)$ maximizes the volume of positive images of $U L'$ contained in $K'$, we must have $\det Q^* \ge \det P^*$ on $O_n$.

By assumption, $1 = \det I = \min_U \det P^*(U)$; since $(x_i, y_i)$ are contact pairs of $K', L'$ satisfying Equations \eqref{iso_trans}, \eqref{iso_dil} and positive John position is characterized by these conditions, we have $(Q^*(I), w^*(I)) = (I, 0)$ as well, and in particular, $\det Q^*(I) = 1$. Since $\det Q^* \ge \det P^*$, we have $1 = \min Q^*$, so that $L'$ is in saddle-John position inside $K'$. We claim that $K', L'$ satisfy the hypotheses of Theorem \ref{pos_john_sm}. First, since $L' \subset L \subset K'$,
\begin{equation}
\{x_1, \ldots, x_k\} \subset \partial L_m \cap \partial L \subset \partial L_m \cap \partial K = \{x_1, \ldots, x_k\}.
\end{equation}
As each point $x_i \in \partial K' \cap \partial L'$ uniquely determines the contact pair $y_i = \frac{n_K(x_i)}{h_K(x_i)}$ because $K'$ is $C^2_+$, we see that the only contact pairs of $K'$ and $L'$ are the $(x_i, y_i)$, and by our choice of $(x_i, y_i)$, the weights $c_i$ such that $\sum c_i ((x_i \otimes y_i)_{sym}, y_i) = (I, 0)$ are unique. Hence, by the first part of the proof, $\sum c_i (x_i \otimes y_i) = I_n$, and we are done.
\end{proof}

\begin{remark}Examining the proof shows that the statement of the theorem can actually be mildly strengthened: if $L$ is in saddle-John position inside $K$ and $\{(x_i, y_i)\}_{i = 1}^k$ is a minimal set of contact pairs such that $\sum c_i ((x_i \otimes y_i)_{sym}, y_i) = (I, 0)$ then necessarily $\sum c_i (x_i \otimes y_i) = I$.
\end{remark}

It follows immediately from the definition that if $L' = P^*(U) UL + z^*(U)$ is a saddle-John image of $L$ inside $K$ (i.e., $(P^*(U), z^*(U)) = (P, z)$ and $\det P^*(V) \ge \det P$ for all $V \in O_n$), then $L'$ is a saddle-John image of $VL$ for every $V \in O_n$. However, as for positive John position, it is not necessarily the case that $L'$ is itself in saddle-John position in $K$. By analogy with Proposition \ref{prop:pos_john_im}, one might expect that $L'' = P^*(U)^{\frac{1}{2}} L + P^*(U)^{-\frac{1}{2}} z^*(U)$ is in saddle-John position inside $P^{-\frac{1}{2}} K$. We do not know how to prove this, but we can prove that $L''$ shares with the saddle-John position the property of supporting a genuine decomposition of the identity:

\begin{proposition}Suppose that $L' = P UL$ is a saddle-John image of $L$ inside $K$, and let $L'' = P^{-\frac{1}{2}} L'$, $K' = P^{-\frac{1}{2}} K$. Then $L''$ is in positive-John position inside $K'$; moreover, there exist contact pairs $(x_1, y_1), \ldots, (x_m, y_m)$  of $K', L''$ and $c_1,\ldots,c_m > 0$ such that:
\begin{enumerate}
\item $\sum_{i = 1}^m c_i y_i = 0$.
\item $\sum_{i = 1}^m c_i (x_i \otimes y_i) = I_n$.
\end{enumerate}
\begin{proof}We will prove this in the case where $K, L$ satisfy the hypotheses of Theorem \ref{pos_john_sm}; the extension to general pairs is similar to the argument of the previous proof and left to the reader.

The first statement is just Proposition \ref{prop:pos_john_im}. By Theorem \ref{pos_john_contact}, we thus get contact pairs $(x_i', y_i')$ of $K', L''$, along with weights $c_i$, such that $\left(\sum c_i (x_i' \otimes y_i')\right)_{sym} = I_n$. Write $(x_i', y_i') = (P^{\frac{1}{2}} x_i,  P^{\frac{1}{2}} y_i)$; noting that $(P^{-\frac{1}{2}} K)^\circ = P^{\frac{1}{2}} K^\circ$, we see that the $(x_i, y_i)$ are contact pairs of $UL$ and $K$. Letting $M = \sum c_i x_i \otimes y_i$ we see that $(P^{\frac{1}{2}} M P^{\frac{1}{2}})_{sym} \propto I$; equivalently, $M_{sym} \propto P^{-1}$. 

Next, applying the envelope theorem to $UL$ and $K$ as in the proof of Theorem \ref{saddle_john_contact}, and using the fact that $x_i$ and $y_i$ are contact pairs, we obtain that for any antisymmetric $A$, 
$$0 = \left.\frac{d}{dt}\right|_{t = 0} \sum c_i g_K(P e^{tA} x_i) = \sum c_i \langle y_i, P A x_i\rangle = \sum c_i \langle P y_i, A x_i\rangle$$ 
(the first equality follows from noting, as before, that $y_i \propto n_K(x_i)$), which implies that $\left(\sum c_i x_i \otimes y_i\right) P = MP$ is symmetric. We claim, moreover, that $M$ and $P$ commute. Assuming this, $M = MP \cdot P^{-1}$ is the product of two symmetric commuting matrices, hence symmetric, so $M = M_{sym} \propto P^{-1}$ and thus $P^{\frac{1}{2}} M P^{\frac{1}{2}} =  \sum c_i x_i' \otimes y_i'$ is proportional to $I_n$, and by simply scaling the $c_i$ we may obtain $I_n$, as desired.

It remains to prove that $M$ and $P$ commute. Let $S = [P, M]$ be their commutator; we have $S = PM - MP = M^T P - P M^T = (PM - MP)^T = S^T$. In addition, $[P, MP] = [P, M] P = SP$, the product of a symmetric matrix and a positive-definite matrix, which must therefore have real eigenvalues (as in the proof of Lemma \ref{gen_pol_decomp}); but $P$ and $MP$ are symmetric, and the commutator of symmetric matrices is obviously antisymmetric, and thus has imaginary eigenvalues. The only way both statements can hold is if $S = 0$, and we are done. 
\end{proof}{}
\end{proposition}

Gordon, Litvak, Meyer and Pajor \cite[Theorem 5.1]{GLMP} showed that for any two convex bodies $K, L$, there exists a translation $K'$ of $K$ and an affine image $L'$ of $L$ such that $L' \subset K' \subset -n L'$. These are obtained as follows: first, one finds a maximal volume image $L''$ of $L$ in $K$, and then translates $L''$ and $K$ according to the formulas we gave in the remark following the proof of Theorem \ref{pos_john_contact} (just before \S \ref{subsection:set_valued}); the main observation is that the existence of contact pairs $(x_i, y_i)$ and weights $c_i$ such that $\sum c_i x_i = \sum c_i y_i = 0$ and $\sum c_i x_i \otimes y_i = Id$ for $L \subset K$ guarantees that $K \subset -nL$. Since saddle-John images also yield decompositions of the identity, the same argument shows that (up to translations), $K$ is contained in $-n L_{saddle}$; for completeness, we shall give the proof, which directly follows \cite{GLMP}. 

\begin{proposition}\label{saddle_dilate} Suppose that $L_s = PUL + z$ is the saddle-John image of $L$ inside $K$. Then there exists $a \in \mathbb R^n$ such that $K - a \subset -n(L_s - a)$.
\begin{proof}By the preceding proposition, the pair $L' = P^{-\frac{1}{2}} L_s$, $K' = P^{-\frac{1}{2}} K$ supports a decomposition of the identity with contact pairs $(x_i, y_i)$ and weights $c_i$, and by the remark following the proof of Theorem \ref{pos_john_contact}, by translating $L'$ and $K'$ we may assume $\sum c_i x_i = \sum c_i y_i = 0$. It suffices to show that under these conditions, $K' \subset -n L'$; multiplying $K'$ and $L'$ by $P^{\frac{1}{2}}$ will then yield the desired conclusion.

Let $x \in K'$; we wish to show $x \in -nL'$. Since $\sum c_i (x_i \otimes y_i) = I_n$, we have $x = \sum c_i x_i \langle y_i, x\rangle$. Since $\sum c_i x_i = 0$, the RHS also equals $\sum c_i (1 - \langle y_i, x\rangle) (-x_i)$; as $y_i \in (K')^\circ$, $\langle y_i, x\rangle \le 1$, all the coefficients are positive and so 
$$\sum c_i (1 - \langle y_i, x\rangle) (-x_i) \in \sum c_i (1 - \langle y_i, x\rangle) (-L')$$
by convexity, as $-x_i \in -L'$. But 
\begin{equation}
\sum c_i (1 - \langle y_i, x\rangle) = \sum c_i - \left\langle \sum c_i y_i, x \right\rangle = \sum c_i = n
\end{equation}
because $\sum c_i y_i = 0$, so we obtain $x \in -n L'$, as desired.
\end{proof}
\end{proposition}

\subsection{Examples}\label{subsec:saddle_examples}
Given a pair of bodies $K, L$, it is interesting to compare the volumes of the maximal-volume and saddle-John positions of $L$ inside $K$. It is also natural to ask whether the position of maximal volume or the saddle-John position of $L$ inside $K$ is more ``typical'': more precisely, one can ask whether the volumes of the positive John family of $L$ inside $K$ are close to the maximal volume ``most of the time'' (in the sense of Haar measure on $O_n$), close to the minimal volume, or perhaps neither. We give three examples showcasing various kinds of behavior.

First, let $K = B^n_\infty$, $L = B^n_1$. It's clear that $L$ is in saddle-John position, as for any $U$ we have $UL \subset B^n_2 \subset K$. On the other hand, it is well-known that for any $n$ there exists an orthonormal basis $\{u_i\}_{i = 1}^n$ such that $\max_i \max_j |\langle u_i, e_j\rangle| \le \frac{2}{\sqrt n}$, which (following \cite{K}), we call a Walsh basis. In particular, for $n = 2^m$, the usual Hadamard-Walsh basis of $\mathbb R^n$ satisfies $|\langle u_i, e_j\rangle| = \frac{1}{\sqrt n}$ for all $i, j$, which is optimal. $L' = \conv \{\pm u_i\}$ is an orthogonal image of $L$ which clearly satisfies $\frac{\sqrt n}{2} L' \subset K$, and we claim that this is asymptotically optimal: i.e., if $L_{max\text{-}vol} = M L$ is a maximal volume image of $L$, then $\det(M)^{\frac{1}{n}} = O(\sqrt n)$. Indeed, let $M \in GL_n$ be arbitrary and let $V D U$ be its singular value decomposition; then $M L \subset K$ is equivalent to $D(UL) \subset V^T K$. Letting $u_i = U e_i$, $v_i = V^T e_i$, this is equivalent to the condition
\begin{equation}
\max_{ij} d_i |\langle u_i, v_j\rangle| \le 1.
\end{equation}
Let $i_0 = \argmax d_i$; we have $\sum_j \langle u_{i_0}, v_j\rangle^2 = 1$ and so $\max_j |\langle u_{i_0}, v_j\rangle| \ge \frac{1}{\sqrt n}$, implying $d_i \le \sqrt n$. Hence, $\left(\frac{|L_{max\text{-}vol}|}{|L_{saddle}|}\right)^{\frac{1}{n}} \le \sqrt n$ and up to a factor of $2$ (which we can do without if $n = 2^m$), this is achieved by a rotation to a Walsh basis followed by a dilation.

What about a random orthogonal basis? It's well-known that if $U$ is a random orthogonal matrix, with high probability every entry of $U$ satisfies $|u_{ij}| \le \sqrt{\frac{\log n}{n}}$. We'll copy the simple proof from \cite[Lemma 2.3]{K}: each row of $U$ is a uniformly distributed vector on $S^{n - 1}$, and Lipschitz concentration on the sphere applied to the function $x \mapsto x_j$ yields 
\begin{equation}
\mathbb P(|u_{ij}| \ge \epsilon) \le c e^{-\frac{\epsilon^2 n}{2}}
\end{equation} for any $i, j$ and $\epsilon$. Taking $\epsilon = 20 \sqrt{\frac{\log n}{n}}$, say, we obtain that $|u_{ij}| \le 20 \sqrt{\frac{\log n}{n}}$ with probability $1 - n^{-10}$, and a union bound yields that 
\begin{equation}
\mathbb P\left(\max_{i, j} |u_{ij}| < 20 \sqrt{\frac{\log n}{n}}\right) \ge 1 - n^{-8}.\end{equation}
Thus $\frac{1}{20}\sqrt{\frac{n}{\log n}} U B^n_1 \subset B^n_\infty$ with high probability, i.e., the positive John family of $B^n_1$ is usually close to its position of maximal volume, up to a logarithmic factor.

Dually, take $L = B^n_\infty$, $K = B^n_1$. We have $\frac{1}{n} L \subset K$, with contact pairs $\{(\frac{1}{n} \epsilon, \epsilon): \epsilon \in \{\pm 1\}^n\}$; it's clear from symmetry that these support a decomposition of the identity, so this is the positive John image of $L$ in $K$. Moreover, $\frac{1}{n} L$ is a saddle-John position: indeed, $\frac{1}{n} L \subset \frac{1}{\sqrt n} B^n_2 \subset K$, so $\frac{1}{n} UL \subset K$ for any $U \in O_n$, implying that all images in the positive John family of $L$ inside $K$ have at least the volume of $\frac{1}{n} L$. On the other hand, 
\begin{equation}
\frac{\vol(K)}{\frac{1}{n} \vol(L)} = \frac{2^n}{n!} \cdot \left(\frac{2}{n}\right)^{-n} \sim \frac{e^n}{\sqrt{2\pi n}}
\end{equation}
by Stirling's approximation, which means that any image of $L$ inside $K$ has volume at most $(e + o(1))^n$ times the volume of the saddle-John image. In this case, then, the disparity between saddle-John position and position of maximal volume is not too significant.

As a third example, consider $K = L = B^n_\infty$. Obviously, $B^n_\infty$ is in a position of maximal volume inside itself. On the other hand, suppose $n = 2^m$, and let $U$ be an orthogonal matrix sending some Hadamard-Walsh basis of $\mathbb R^n$ to the standard basis. Then $UB^n_\infty$ contains the vectors $\pm \sqrt n e_1, \ldots, \pm \sqrt n e_n$ and is contained in $\sqrt n B^n_2$, so the standard basis yields a set of contact pairs for $\frac{1}{\sqrt n} UB^n_\infty \subset B^n_\infty$. Hence $\frac{1}{\sqrt n} UB^n_\infty$ is the positive John position of $UB^n_\infty$ inside $B^n_\infty$; since $V B^n_\infty \subset \sqrt n B^n_\infty$ for any $V \in O_n$, $\frac{1}{\sqrt n} U B^n_\infty$ is obviously a saddle-John position of $B^n_\infty$ inside itself. Conversely, one easily sees that any saddle-John position of $B^n_\infty$ is obtained by this construction from some Hadamard basis of $\mathbb R^n$, i.e., a set of $n$ orthogonal vertices of $B^n_\infty$. (More precisely, this holds in any dimension $n$ for which there exists a Hadamard basis; we do not know how to characterize the saddle-John position of $B^n_\infty$ in itself in other dimensions.)

Finally, we claim that in any dimension, $\det P^*(U)^{\frac{1}{n}} = O(\sqrt{\frac{\log n}{n}})$ with high probability over $U \in O_n$ (recall that $P^*(U)$ is defined such that $P^*(U) UB^n_\infty$ is the positive John image of $UB^n_\infty$ in $B^n_\infty$). First, note that if $\alpha B^n_1 \subset UB^n_\infty$ for some $\alpha > 0$ then $\det P^*(U)^{\frac{1}{n}} \le \alpha^{-1}$. Indeed, $P^*(U) (\alpha B^n_1) \subset P^*(U) UB^n_\infty \subset B^n_\infty$, i.e., $P^*(U) (\alpha B^n_1)$ is a positive image of $\alpha B^n_1$ contained in $B^n_\infty$, and so $\vol(P^*(U) (\alpha B^n_1)) = \alpha^n \det P^*(U) \vol(B^n_1)$ is bounded above by the volume of the positive John image of $\alpha B^n_1$ in $B^n_\infty$. But we know, by the above, that this image is precisely $B^n_1$, which means that $ \alpha^n \det P^*(U) \le 1$, as claimed. So it is sufficient to show that $c \sqrt{\frac{n}{\log n}} B^n_1 \subset U B^n_\infty$ with high probability over $U$ for some absolute constant $c$; but this is precisely what we showed above when considering the positive John family of $B^n_1$ in $B^n_\infty$. In other words, the positive John family of $B^n_\infty$ in itself is usually ``closer'' to the saddle-John position, in terms of volume, than to the position of maximal volume.

\section{Positive John images inside ellipsoids}\label{section:pos_john_ellips}
If $K$ or $L$ is a Euclidean ball, the family of positive John images of $L$ inside $K$ is of course ``trivial'': for $L = B^n_2$, $P^*(U), z^*(U)$ do not depend on $U$ at all, and if $K = B^n_2$, $P^*$ and $z^*$ vary formulaically with $U$: if $(P^*(I), z^*(I)) = (P_0, z_0)$, then $(P^*(U), z^*(U)) = (U P_0 U^T, Uz)$, so that $P^*(U) UL + z^*(U) = U (P_0 L + z_0)$. In particular, the volume of each of the positive John images in either of these cases is independent of $U$.

Since ellipsoids are related to the Euclidean ball by a positive transformation, it is reasonable to expect that the positive John family of a convex body $L$ inside an ellipsoid $E$, or of an ellipsoid $E$ inside a convex body $L$, is also trivial in a similar sense. This expectation will be partially vindicated: in the course of the section, we shall see that in this case, $P^*$ and $z^*$ vary predictably with $U$, and that $\det P^*$ is constant, but the formula defining $P^*$ in terms of $U$ turns out to be rather complicated.

In the sequel, we will treat only the positive John family of a convex body $L$ inside an ellipsoid $E$; the case of the positive John family of an ellipsoid inside a convex body $K$ involves the same ideas. In addition, for simplicity, we shall assume $L$ is centrally symmetric, obviating the need to deal with translations; the extension to the non-symmetric case is routine.

\begin{proposition}\label{prop-PJPin-ellipsoid}
	Let $P \in \mathcal P^n$ be a positive matrix, $E = P B^n_2$ the corresponding 
	ellipsoid, and $L \in \mathcal K^n_s$ an arbitrary centrally symmetric convex body. Then $E$ is in positive John position with respect to $L$ if and only if $B^n_2$ is in positive John position with respect to $P^{-1} L$, i.e., $P^{-1} L$ is in L\"owner position. In particular, all the bodies in the positive John family of $L$ inside $E$ have the same volume.
\end{proposition}

\begin{proof}First, note that the polar body of $E$ is $P^{-1} B^n_2$, and for any $x \in \partial E$, which can be expressed as $Pu$ for $u \in S^{n - 1}$, the unique point $y \in \partial E^\circ$ such that $\langle y, x\rangle = 1$ is given by $y = P^{-1} u$.

Suppose that $E$ is in positive John position with respect to $L$. By Theorem \ref{pos_john_contact}, there exist contact pairs $(x_i, y_i) \in (\partial E \cap \partial L) \times (\partial E^\circ \times \partial L^\circ)$ such that $\sum_{i = 1}^m c_i (x_i \otimes y_i)_{sym} = I_n$. Let $x_i = P u_i$, $y_i = P^{-1} u_i$ for $u_i \in S^{n - 1}$; then
\begin{equation}
I = \sum_{i = 1}^m c_i (x_i \otimes y_i)_{sym} = \sum_{i = 1}^m c_i (Pu_i \otimes P^{-1} u_i)_{sym} = \left(P \left(\sum_{i = 1}^m c_i (u_i \otimes u_i)\right) P^{-1} \right)_{sym}
\end{equation}
We claim that $M = \sum_{i = 1}^m c_i (u_i \otimes u_i)$ is itself equal to $I$. Indeed, $\Tr M = n$, so if the positive-definite symmetric matrix $M$ does not equal $I$ then it has an eigenvector with eigenvalue $\lambda > 1$, hence so does $P M P^{-1}$; letting $v$ be such an eigenvector, we have
\begin{equation}\langle v, v\rangle = \langle v, (P M P^{-1})_{sym}) v\rangle = \langle v, P M P^{-1} v\rangle = \lambda \langle v, v\rangle,
\end{equation}
contradiction. Hence $M = I$.

It remains only to note that if $(x_i, y_i) = (P u_i, P^{-1} u_i)$ are contact points of $E = P B^n_2$ and $L$, then $(u_i, u_i)$ are contact points of $B^n_2$ and $P^{-1} L$; since $\sum_{i = 1}^m c_i (u_i \otimes u_i) = I_n$, the standard John's theorem yields that $P^{-1} L$ is in L\"owner position, as desired.

Thus, the positive John family of a body $L$ inside $E$ may be described as $\{P U L': U \in O_n\}$ for any L\"owner position $L'$ of $L$. The last statement follows immediately by noting that all the L\"owner positions of a body are related by orthogonal transformations, and in particular have the same volume.
\end{proof}
 
This proposition enables us to give an explicit formula for $P^*(U)$ given $P_0 = P^*(I)$:

\begin{corollary}\label{john_ellips_char} Under the same hypotheses, let $P_0 L$ be the positive John image of $L$ inside $E$. Then for any $U \in O_n$, $P^*(U)$ is the unique positive matrix for which there exists $V \in O_n$ such that $P^*(U) P^{-1} V = U P_0 P^{-1}$.
\begin{proof}If $P^*(U) U L$ is the positive John position of $UL$ inside $E$, by the proposition, $P^{-1} P^*(U) U L$ and $P^{-1} P_0 L$ are L\"owner positions of $L$, so $P^{-1} P^*(U) U L = V P^{-1} P_0 L$ for some $V \in O_n$. This implies that
\begin{equation}\label{john_ellips_formula}
P^{-1} P^*(U) U = V P^{-1} P_0 A
\end{equation}
for some $A$ in the symmetry group of $L$. The symmetry group of a convex body must preserve the minimal-volume ellipsoid containing the body; as $L' = P^{-1} P_0 L$ is in L\"owner position, its symmetry group must be contained in $O_n$. In addition, for any set $S$ and $M \in GL_n$, $\Aut(MS) = M \Aut(S) M^{-1}$, so putting these facts together we get $A = P_0^{-1} P W P^{-1} P_0$ for some $W \in O_n$. Substituting in Equation \eqref{john_ellips_formula}, we see that $P^*(U) U L$ is a positive John position of $UL$ inside $E$ if and only $P^{-1} P^*(U) U = V W P^{-1} P_0$ for some $V \in O_n$, $W \in \Aut(L')$, and since $\Aut(L') \subset O_n$ we may absorb $W$ into $V$. Rearranging yields $P^*(U) P^{-1} V = U^T P_0 P^{-1}$, an equation of the form treated in Lemma \ref{gen_pol_decomp}, for which we know there exists a unique solution, explicitly given by Equation \eqref{gen_sq_root} (though the ensuing formula is not very enlightening).
\end{proof}
\end{corollary}

\begin{remark}
It is interesting to note that Lemma \ref{gen_pol_decomp} is not strictly necessary for the proof of Corollary \ref{john_ellips_char}: indeed, since the existence and uniqueness of $P^*(U)$ follow from earlier results, it is only necessary to show that $P^*(U) UL$ is a positive John position if and only if there exists $V \in O_n$ such that $P^*(U) P^{-1} V = U P_0 P^{-1}$. Existence and uniqueness of the matrix $P^*(U)$ satisfying this condition -- from which it is trivial to derive the full statement of Lemma \ref{gen_pol_decomp} -- are thus obtained as a corollary of the existence and uniqueness of positive John position.
\end{remark}

We conclude the section with the following conjecture, which states that the behavior exhibited by ellipsoids with regard to positive John position is unique to that class:

\begin{conjecture}Suppose that $K, L$ are convex bodies, neither of which is an ellipsoid. Then the positive John family $(P^*, z^*)$ associated to $K, L$ does not satisfy $\det P^* = C$ identically.
\end{conjecture}

We can confirm this conjecture in the case $K = L$. Indeed, in this case we clearly have $\det P^*(U) \le 1$ for all $K$, with equality if and only if $P^*(U) U K + z^*(U) = K$, i.e., $x \mapsto P^*(U) U x + z^*(U)$ is a symmetry of $K$; by taking $K$ to be centered, we may assume $z^*(U) = 0$ for all $U$. If $K$ is not an ellipsoid then its symmetry group has dimension smaller than that of $O_n$, so the set of $U$ such that there exists $P$ for which $PU \in \mathrm{Aut}(K)$ is also lower-dimensional, since the polar decomposition is a diffeomorphism.

\section{Maximal intersection position}\label{section:max_int_pos}
Another way to generalize the John position, introduced by Artstein-Avidan and Katzin \cite{AK}, is to consider affine images of the ball not necessarily contained in $K$. They studied the following question: given a convex body $K$ and a prescribed volume $V$, what can be said about the ellipsoid maximizing $\mathcal E \cap K$ over all ellipsoids with volume $V$? If $V$ is taken to be the volume of the John ellipsoid of $K$, then the maximizing ellipsoid is the John ellipsoid, and similarly for the L\"owner ellipsoid; but of course one does not know which choices of $V$ will yield the John or L\"owner ellipsoid of $K$ unless one already has these ellipsoids in hand.

Artstein-Avidan and Katzin considered this question for symmetric convex bodies $K$. They showed the existence of an ellipsoid $\mathcal E$ maximizing the volume $\vol(\mathcal E \cap K)$ over all ellipsoids of given volume, and defined $K$ to be in maximal intersection position of radius $r$ if this ellipsoid is the ball $r B^n_2$. Their main result was the following:

\begin{theorem}\label{max_ak} Let $K\subset\mathbb R^n$ be a centrally symmetric convex body such that:
\begin{enumerate}
  \item $\mathcal H^{n - 1}(\partial K \cap \partial\mathcal E) = 0$ for all but finitely many ellipsoids $\mathcal E$,
  \item $\mathcal H^{n - 1}(\partial K \cap r S^{n - 1}) = 0$, and
  \item $\mathcal H^{n - 1}(K \cap r S^{n-1}) > 0$.
\end{enumerate}
  If $K$ is in maximal intersection position of radius $r$, then the restriction $\mu$ of the surface area measure on the sphere to $S^{n - 1} \cap r^{-1} K$ is isotropic.
\end{theorem}

In the following, we give a generalization of maximal intersection position along the lines of the generalization of John position to the position of maximal volume. The generalization is twofold: firstly, we allow not-necessarily centrally symmetric bodies, and in addition, we consider general pairs of convex bodies, not just a convex body and a ball:

\begin{definition}[Maximal intersection position] Let $K, L \subset \mathbb R^n$ be convex bodies. We say that $K, L$ are in maximal intersection position if for every $A \in SL_n, z \in \mathbb R^n$ and $L' = AL + z$, we have $\vol(K \cap L') \le \vol(K \cap L)$; clearly, this definition is symmetric with respect to an interchange of $K$ and $L$.
\end{definition}

We first show that for given $K, L$, the maximal intersection position exists, i.e., there exists an image $L'$ of $L$ with $\vol(L') = \vol(L)$ satisfying
\begin{equation}\label{max_int_cond}
\vol(L' \cap K) = \sup \{\vol((AL + z) \cap K)): A \in SL_n, z \in \mathbb R^n\}
\end{equation}
Write $m_{K, L}$ for the RHS of \eqref{max_int_cond}, and consider a sequence of convex bodies $L_j = A_j L + z_j$ with ${A_j \in SL_n}$ and $\vol(L'_j \cap K) \to m_{K, L}$. First, suppose for the sake of contradiction that some coordinate of $A_j$ goes to $\infty$; we have
\begin{equation}
\vol(L_j \cap K) \le \vol((A_j RB^n_2 + z_j) \cap RB^n_2)
\end{equation}
where $R > 0$ is sufficiently large so that $RB^n_2$ contains both $L$ and $K$. We use the singular value decomposition to write $A_j = U_j \Sigma_j V_j$ for $\Sigma_j$ diagonal with positive, decreasing entries on the diagonal and $U_j, V_j$ orthogonal. The maximal entry of $\Sigma_j$ must go to $\infty$, as the orthogonal group $O_n$ and the set of diagonal matrices with entries bounded by $c$ are compact. We have
\begin{equation}
\vol((U_j \Sigma_j V_j RB^n_2 + z_j) \cap RB^n_2) = \vol((\Sigma_j RB^n_2 + z_j') \cap RB^n_2)
\end{equation}
where $z_j' = U_j^L z_j V_j^L$. Since $(\Sigma_j)_{nn} \to 0$ as $j\to \infty$ and since $\vol((\Sigma_j RB^n_2 + z_j') \cap RB^n_2)$ is bounded by the maximal volume of a slice of width $(\Sigma_j)_{nn}$ of $RB^n_2$, which goes to $0$, so we must have that as $j\to \infty$
\begin{equation}
\vol(L_j \cap K) \le \vol((\Sigma_j RB^n_2 + z_j') \cap RB^n_2) \to 0,
\end{equation}
a contradiction.

Thus the coordinates of the matrices $A_j$ are bounded, so  $A_j L$ are all contained within a given compact set, and hence the $z_j$ also must be bounded, as otherwise $A_j L + z_j$ will not intersect $K$ at all for large enough $j$. By compactness, we obtain a subsequence $(A_{j_k}, z_{j_k})$ of the $(A_j, z_j)$ converging to $(A, z)$, which give the desired maximizer $AL + z$.

Hence, for any $K, L$, there exists an affine image $L'$ of $L$ such that $K, L'$ are in maximal intersection position. Our main result, like the general John's theorem (Theorem \ref{gen_john_thm}), gives a decomposition of the identity associated to this position:

\begin{manualtheorem}{\ref{gen_max_int_pos}} Let $K, L \subset \mathbb R^n$ be convex bodies, and suppose that $K, L$ are in maximal intersection position and that $\vol_{n-1}(\partial K\cap\partial L) = 0$. For any $x \in \partial L$, let $\hat n_L(x)$ be the unit normal at $x$, which is defined $\mathcal H^{n - 1}$-almost everywhere on $\partial L$. Then we have
\begin{align}
\int_{K \cap \partial L} \hat n_L(x)\,d\mathcal H^{n - 1}           &= 0          \label{mint_pos_eq1}, \\
\int_{K \cap \partial L} x \otimes \hat n_L(x)\,d\mathcal H^{n - 1} &\propto I_n  \label{mint_pos_eq2}.
\end{align}
The same formulae hold when interchanging the roles of $K, L$.
\end{manualtheorem}

Along with generalizing Theorem \ref{max_ak}, this theorem also strengthens it: it is unnecessary to assume $\vol_{n-1}(\partial K\cap\partial L') = 0$ for all but finitely many affine images $L'$ of $L$.

The strategy of proof is to consider $f(t) = \vol(K \cap L_t)$ for a one-parameter family $L_t$ of affine images of $L$ with $L_0 = L$, and show that $f$ is differentiable at $0$ and $f'(0) = 0$. In order to show differentiability, the proof of Theorem \ref{max_ak} in \cite{AK} used explicit approximations of the indicator functions of the body $K$ and of the ball, $1_K, 1_{B^n_2}$, by smooth functions $\psi_k, \varphi_k$. We shall give two proofs: the first avoids the need for any approximation procedure, but treats translations and linear transformations separately; the second requires approximation (though not by explicitly-constructed functions), and also has the restrictive assumption of Theorem \ref{max_ak}, but it treats a much more general family of transformations (see Theorem \ref{diff_int} for details).

\subsection{A hand-waving argument}
Before we proceed to rigorous proofs, we give a hand-waving argument for Theorem \ref{gen_max_int_pos} based loosely on the theory of distributions, which shows why we should, intuitively, expect the theorem to be true.

First, suppose $K, L$ are in maximal intersection position with respect to translations. For any $x \in \mathbb R^n$, $\vol(K \cap (x + L)) = \int 1_K(y) 1_L(y - x)\,dy$, so non-rigorously, we may write
\begin{equation}
\nabla_x \vol(K \cap (x + L)) = \int 1_K(y) (-\nabla 1_L)(y - x)\,dy.
\end{equation}
But by the vector calculus identity
\begin{equation}
\int_L \nabla f\,dx = \int_{\partial L} f\hat n_L\,d\mathcal H^{n - 1},
\end{equation}
we have that $(-\nabla 1_L)$, considered as a distribution, is just the normal vector to $\partial L$ times a one-dimensional delta function supported on $\partial L$, so we obtain
\begin{equation}\left.\nabla_x \vol(K \cap (x + L))\right|_{x = 0} = \int 1_K(y) \hat n_L(y) \delta_{\partial L}(y) \,dy = \int_{K \cap \partial L} \hat n_L(y)\,dy.
\end{equation}
As $x = 0$ is the maximizer of $\vol(K \cap (x + L))$, the gradient vanishes.

This isn't a proof, of course, because distributions can only be integrated (a priori) against smooth test functions, not against functions like $1_K$. However, one can hope that if the discontinuity of $1_K$ is ``transverse'' to the discontinuity of the delta function supported on $\partial L$ then the results of the computation can be shown to be valid by some approximation procedure. It seems intuitively clear that if $\vol_{n - 1}(\partial K \cap \partial L) > 0$, this procedure has no chance of working, and indeed the theorem does not hold in this case.

A similar ``argument'' can be used to justify the formula for maximal intersection with respect to volume-preserving linear transformations. A local perturbation of the identity within $SL_n$ looks like $A(t) = e^{t M}$ for some matrix $M$ with $\Tr\,M = 0$ (i.e., the Lie algebra of $SL_n$ is the space of traceless matrices), so we have
\begin{align}\left.\frac{d}{dt}\right|_{t = 0}\vol(K \cap A(t) L) &= \int 1_K(x) \left.\frac{d}{dt}\right|_{t = 0} 1_L(e^{-t M} x)\,dx  \nonumber \\
 &= \int 1_K(x) \nabla 1_L \cdot \left.\frac{d}{dt}\right|_{t = 0} (e^{-tM} x)\,dx  \nonumber \\
 &= \int 1_K(x) \left(-\delta_{\partial L}(x) \hat n_L(x) \cdot (-Mx)\right) \nonumber \\
 &= \int_{K \cap \partial L} \langle \hat n_L(x), Mx\rangle\,dx.
\end{align}

Again, if we believe this non-rigorous calculation, we obtain that for $K, L$ in maximal intersection position, $\int_{K \cap \partial L} \langle \hat n_L(x), Mx\rangle\,dx = 0$, or in other words that $\int_{K \cap \partial L} \hat n_L(x) \otimes x\,dx$ is Hilbert-Schmidt orthogonal to $M$. It is not hard to see that a matrix $A$ is Hilbert-Schmidt orthogonal to all traceless matrices if and only if it is proportional to the identity, so we get that $\int_{K \cap \partial L} \hat n_L(x) \otimes x\,dx \propto I_n$, as desired.

We now proceed to the actual proofs.

\subsection{First proof of Theorem \ref{gen_max_int_pos}}

Theorem \ref{gen_max_int_pos} is a consequence of the following two differentiation formulae:

\begin{theorem}\label{thm:derivative-of-volume-translate}Let $K, L\subset\mathbb R^n$ be convex bodies. Let $u\in S^{n - 1}$ and let $V(t):\mathbb{R}\rightarrow\mathbb{R}$ be defined by $V(t)=\vol_{n}(K\cap (L + tu))$. If $\vol_{n-1}(\partial K\cap\partial L) = 0$, then
\begin{equation}\label{trans_der}
\left.\frac{dV(t)}{dt}\right|_{t=0}=\int_{\partial L\cap K}\left\langle \hat n_L(x), u\right\rangle d\sigma_L(x)
\end{equation}
where $\hat n_L(x) \,d\sigma_L(x)$ is the vector surface area measure on $\partial L$.
\end{theorem}

\begin{theorem}\label{thm:derivative-of-volume-linear}Let $K, L\subset\mathbb{R}^{n}$ be convex bodies with $0 \in \intr(K \cap L)$. Fix $A\in M_{n}(\mathbb{R})$, and let $V(t): \mathbb R \rightarrow \mathbb R $
be defined by $V(t) = \vol_n(K\cap e^{tA} L)$. If $\vol_{n-1}(\partial K \cap \partial L) = 0$, then
\begin{equation}\label{lin_der}
\left.\frac{dV(t)}{dt}\right|_{t=0}=\int_{\partial L\cap K}\left\langle \hat n_L(x),Ax\right\rangle d\sigma_L(x)
\end{equation}
where $\hat n_L(x) \,d\sigma_L(x)$ is the vector surface area measure on $\partial L$.
\end{theorem}

\begin{proof}[Proof of Theorem \ref{gen_max_int_pos}]
Indeed, as we have already argued above, if $K$ and $L$ are in maximal intersection position then
\begin{equation}
\frac{d}{dt}{\rm vol}_{n}(K\cap (L + tu)) = \int_{\partial L\cap K}\left\langle \hat n_L(x),u\right\rangle d\sigma_L(x)
\end{equation}
must vanish for any $u$, so $\int_{\partial L\cap K}\hat n_L(x) d\sigma_L(x) = 0$; furthermore, for any $A$ with trace zero, such that $e^{tA} \in SL_n$,
\begin{equation}
\frac{d}{dt} \vol_{n}(K\cap e^{tA} L) = \int_{\partial L\cap K}\left\langle \hat n_L(x),Ax\right\rangle d\sigma_L(x)
\end{equation}
vanishes, which is equivalent to the condition that $\int_{\partial L\cap K} (n_L(x)\otimes x)\rangle d\sigma_L(x) \propto I_n$.
\end{proof}
To prove the two differentiation theorems rigorously we shall need to differentiate under the integral sign. To this end we shall use the following lemma which follows directly from the dominated convergence theorem.

\begin{lemma}\label{min_deriv} Let $(X, \mu)$ be a measure space, and let $f_t \in L^1(\mu)$ for all $t \in (-\epsilon, \epsilon)$. Suppose $\left.\frac{d}{dt}f_t(x)\right|_{t = 0}$ exists for almost every $x \in X$ and that the family $\{\frac{f_t - f_0}{t}: t \in (-\epsilon, \epsilon)\}$ is dominated by some integrable function $h$.
Then
\begin{equation}
\left.\frac{d}{dt}\right|_{t = 0} \int_X f_t \,d\mu = \int_X \left.\frac{df_t(x)}{dt}\right|_{t = 0}\,d\mu(x)
\end{equation}
\qed
\end{lemma}

\begin{remark}
To illustrate how we shall use the above lemma, consider the family $f_t = \min(g_t, h)$ for $g_t$ differentiable everywhere in $t$. Clearly $\min(g_t, h)$ is differentiable in $t$ when $g_t \neq h$, with the result being $\frac{dg_t}{dt}$ for $g_t < h$ and $0$ for $g_t > h$. In particular,
at a given $t$, if $\mu(\{g_t = h\}) = 0$
then $f_t$ is differentiable in $t$ almost everywhere in $x$. As for the second condition of the lemma, case analysis shows that $\frac{|\min(g_t, h) - \min(g_0, h)|}{t} \le \frac{|g_t - g_0|}{t}$ pointwise, so if $\frac{g_t - g_0}{t}$ is dominated by an integrable function, so is $\frac{f_t - f_0}{t}$. Under these conditions,
\begin{equation}
\frac{d}{dt}\int_X \min(g_t, h) = \int_{\{g_t < h\}} \frac{dg_t}{dt}.
\end{equation}
(In fact we can say a bit more: under the same conditions, at any $t$ the left-hand derivative of $\int_X \min(g_t, h)$ exists and equals $\int_{\{g_t < h\}} \frac{dg_t}{dt}$, and the right-hand derivative exists and equals $\int_{\{g_t \le h\}} \frac{dg_t}{dt}$.)

The same considerations will apply to the slightly more complicated function built of minima and maxima of simple functions which we shall encounter presently.
\end{remark}

\begin{proof}[Proof of Theorem \ref{thm:derivative-of-volume-translate}]
Given $K, L$ convex bodies, and $u\in S^{n-1}$, let $u^\perp \subset \mathbb R^n$ be the subspace orthogonal to $u$ and let $P_{u^\perp}$ denote the orthogonal projection to $u^\perp$. Let $X = P_{u^\perp}(K) \cap P_{u^\perp}(L) \subset u^\perp$, and define the functions $w_K^+, w_K^-: X \to \mathbb R$ as follows:
\begin{equation}
w_K^+(x) = \max \{s: x + su \in K\}, w_K^-(x) = \min\{s: x + su \in K\}
\end{equation}
and similarly for $L$, $K$ and $L + tu$. We have $w_{L + tu}^{\pm} = w_L^{\pm} + t$, so for any $x$, the length of $K \cap (L + tu) \cap (x + \mathbb Ru)$ is precisely
\begin{equation}
|[w_K^-(x), w_K^+(x)] \cap [w_L^-(x) + t, w_L^+(x) + t]|.
\end{equation}
We may then  write
\begin{equation}\label{vol_trans}
\vol(K \cap (L + tu)) = \int_X |[w_K^-, w_K^+] \cap [w_L^- + t, w_L^+ + t]|\,dx.
\end{equation}
The family of functions $f_t = |[w_K^-, w_K^+] \cap [w_L^- + t, w_L^+ + t]|$ can be written as a sum of minima and maxima of $w_K^+, w_K^-, w_L^+, w_L^-$ in several ways, but it will be easiest to simply examine it directly. Clearly, $\frac{|f_t - f_0|}{t} \le 1$ for all $t$; also, if
\begin{equation}
\{w_L^+(x), w_L^-(x)\} \cap \{w_K^+(x), w_K^-(x)\} = \emptyset,
\end{equation}
i.e., none of the endpoints of the intervals coincide, we have
\begin{equation}
\left.\frac{df_t(x)}{dt}\right|_{t = 0} = 1_{\{w_L^+ \in (w_K^-, w_K^+)\}} - 1_{\{w_L^- \in (w_K^-, w_K^+)\}}.
\end{equation}
As for the nondifferentiabilty points, if $w_K^+(x) = w_L^+(x)$ then $x + w_K^+(x)u \in \partial L \cap \partial K$, and similarly for the other pairs of width functions, so $f_t$ is differentiable except on $P_{u^\perp}(\partial L \cap \partial K)$. But $\mathcal H^{n - 1}(\partial K \cap \partial L) = 0$ by assumption, and as $P_{u^\perp}$ is Lipschitz, $P_{u^\perp}(\partial L \cap \partial K)$ is also $\mathcal H^{n - 1}$-null. Hence the assumptions of Lemma \ref{min_deriv} are satisfied, and we obtain
\begin{equation}
\left.\frac{d}{dt}\right|_{t = 0} \vol(K \cap (L + tu)) = \vol_{n - 1}(\{x: w_K^+ > w_L^+ > w_K^- \}) - \vol_{n - 1}(\{x: w_K^+ > w_L^- > w_K^-\}).
\end{equation}
Letting $\partial L^+, \partial L^-$ be the positive and negative sides (with respect to the $u$-direction) of the boundary of $L$, respectively, we see that the first term equals
\begin{equation}
\left(\int_{\partial L^+ \cap K} \hat n_L(x)\,d\mathcal H^{n - 1}(x)\right)\cdot u
\end{equation}
and the second term equals
\begin{equation}
\left(\int_{\partial L^- \cap K} \hat n_L(x)\,d\mathcal H^{n - 1}(x)\right)\cdot u
\end{equation}
which sum to
\begin{equation}
\left(\int_{\partial L \cap K} \hat n_L(x)\,d\mathcal H^{n - 1}(x)\right)\cdot u = \int_{\partial L \cap K} \langle \hat n_L(x), u\rangle \,d\mathcal H^{n - 1}(x)
\end{equation}
as desired.
\end{proof}

To prove Theorem \ref{thm:derivative-of-volume-linear} we will make use of the following formula.

\begin{lemma}\label{surf_push} Let $K \in \mathcal K^n$ be a convex body with $0 \in \intr K$, and let $F: \partial K \to S^{n - 1}$ be defined by $F(x) = \frac{x}{|x|}$. Then the pushforward of the vector surface area measure $\hat n_K\,d\mathcal H^{n - 1}$ on $\partial K$ by $F$ is the measure
\begin{equation}-r_K^{n - 2} \nabla r_K\,d\mathcal H^{n - 1}
\end{equation}
on $S^{n - 1}$, where $r_K$ is the radial function of $K$, $r_K(x) = \max \{s \ge 0: sx \in K\}$ (here $\nabla$ is the usual $n$-dimensional gradient, not the spherical gradient).\end{lemma}

\begin{proof}  We remark that since $r_K = \frac{1}{h_{K^\circ}}$, $r_K$ is differentiable almost everywhere on $S^{n - 1}$, so $\nabla r_K$ makes sense.

Recall that the normal vector $n_K: \partial K \to S^{n - 1}$ is defined $\mathcal H^{n - 1}$-almost everywhere. We first check that the pushforward of $\hat n_K\,d\mathcal H^{n - 1}$ points in the same direction as $-r_K^{n - 2} \nabla r_K\,d\mathcal H^{n - 1}$, i.e., that for $x \in \partial K$, $\nabla r_K(\frac{x}{|x|})$ is parallel to $-n_K(x)$. By definition $r_K = \frac{1}{g_K}$, where $g_K = \|\cdot\|_K$ is the gauge function of $K$, and for $x \in \partial K$ we have $\nabla g_K(x) = \frac{n_K(x)}{h_K(n_K(x))}$ \cite[Eq. 1.39]{S}, so
\begin{equation}\nabla r_K(x) = -g_K(x)^{-2} \frac{n_K(x)}{h_K(n_K(x))} = -\frac{n_K(x)}{\langle x, n_K(x)\rangle}
\end{equation}
where we have used the fact that $g_K(x) = 1$ for $x \in \partial K$. The function $r_K$ is $-1$-homogeneous and so $\nabla r_K$ is $-2$-homogeneous, and hence for $x\in \partial K$
\begin{equation}\nabla r_K\left(\frac{x}{|x|}\right) = -|x|^2 \frac{n_K(x)}{\langle x, n_K(x)\rangle} = -r_K\left(\frac{x}{|x|}\right) \frac{n_K(x)}{\langle \frac{x}{|x|}, n_K(x)\rangle}
\end{equation}
which in particular is parallel to $-n_K(x)$. Thus we can take the dot product of both measures with $-n_K$ and compare the resulting scalar measures, reducing to the claim that the surface area measure on $K$ pushes forward via $F$ to
\begin{equation}
\langle -n_K, r_K^{n - 2} \nabla r_K\rangle\,d\mathcal H^{n - 1} = \frac{r_K^{n - 1}(u)}{\langle u, n_K(r_K(u) u))}\,d\mathcal H^{n - 1}(u)
\end{equation}
on $S^{n - 1}$, where $u = \frac{x}{|x|}$ and so $x= r_K(u)u$. This is precisely \cite[Eq. 4.33]{S}.
\end{proof}

\begin{proof}[Proof of Theorem \ref{thm:derivative-of-volume-linear}] Given convex bodies $K$ and $L$ and $A\in M_{n}(\mathbb{R})$, let $r_K: \mathbb R^n \to \mathbb R^+$ be the radial function of $K$ as in the lemma, and similarly $r_{e^{tA} L}$ and  $r_{K \cap (e^{t A} L)} = \min(r_K, r_{e^{t A} L})$. We have
\begin{equation}\label{vol_dil}
\vol(K \cap (e^{t A} L)) = \frac{1}{n}\int_{S^{n - 1}} \min(r_K^n, r_{e^{t A} L}^n)\,d\sigma
\end{equation}
where $\sigma = \mathcal H^{n - 1}$ is the usual (not normalized) area measure on the sphere. Note that
\begin{equation}
r_{e^{t A} L}(x) = \max \{s: sx \in e^{t A} L\} = \max \{s: s e^{-t A} x \in L\} = r_L(e^{-t A} x).
\end{equation}
Also note that $r_L(e^{-t A} x)$ is differentiable in $t$ almost everywhere on $S^{n - 1}$ because $r_L = \frac{1}{h_L^\circ}$ and the gradient of the support function of a convex body exists $\mathcal H^{n - 1}$-almost everywhere on $S^{n - 1}$.

By assumption, $\vol_{n - 1}(\partial K \cap \partial L) = 0$; since $\{x \in S^{n - 1}: r_K(x) = r_L(x)\}$ is the image of $\partial K \cap \partial L$ under the map $x \mapsto \frac{x}{|x|}$, which is Lipschitz away from $0$, we have $\sigma(\{x \in S^{n - 1}: r_K(x) = r_L(x)\}) = 0$. Hence, by the remark following Lemma \ref{min_deriv}, to apply the lemma to \eqref{vol_dil}, we  need only check that $\frac{|r_L(e^{-t A} x)^n - r_L(x)^n|}{t}$ is dominated by an integrable function. Since $L$ is a convex body with $0 \in \intr L$, there exist $r_-, r_+$ such that $0 < r_- \le r_L(x) \le r_+$, so it is sufficient to show that $\frac{|r_L(e^{-t A} x) - r_L(x)|}{t}$ is dominated. Rewrite this expression as
\begin{equation}
\frac{|h_{L^\circ}(e^{-t A} x)^{-1} - h_{L^\circ}(x)^{-1}|}{t} = \frac{1}{h_{L^\circ}(x) h_{L^\circ}(e^{-t A} x)} \frac{h_{L^\circ}(x) - h_{L^\circ}(e^{-t A} x)}{t}.
\end{equation}
Again, $\frac{1}{h_{L^\circ}(x) h_{L^\circ}(e^{-t A} x)}$ is uniformly bounded, and we need only consider $\frac{h_{L^\circ}(x) - h_{L^\circ}(e^{-t A} x)}{t}$. As $h_{L^\circ}$ is a convex function,   the mean value theorem for convex functions yields that there exists some $z\in \partial h_{L^\circ}(y)$  such that $\frac{h_{L^\circ}(x) - h_{L^\circ}(e^{-t A} x)}{t}  = \langle z, e^{-tA}x - x\rangle $ for some $y \in [x, e^{-t A} x]$, and $\partial h_{L^\circ}(y)$ is the support set of $L^\circ$ at $y$ \cite[Theorem 1.7.4]{S}, which is in particular uniformly bounded. Thus the conditions of Lemma \ref{min_deriv} are satisfied, and we obtain
\begin{equation}
\left.\frac{d}{dt}\vol(K \cap (e^{t A} L))\right|_{t = 0} = \frac{1}{n}\int_{\{r_L < r_K\}} \left.\frac{d}{dt} r_{e^{t A} L}^n\right|_{t = 0}\,d\sigma.
\end{equation}

For all $u \in S^{n - 1}$ such that $\nabla r_L$ exists at $u$, we have
\begin{equation}
\left.\frac{d}{dt} \left(\frac{1}{n} r_{e^{t A} L}^n(u)\right)\right|_{t = 0} = r_L(u)^{n - 1} \left.\frac{d}{dt}\right|_{t = 0} r_L(e^{-t A} u) = -r_L(u)^{n - 1} \langle\nabla r_L(u), Au\rangle
\end{equation}
so we obtain
\begin{equation}
\frac{d}{dt}\vol(K \cap (e^{t A} L)) = -\int_{\{r_L < r_K\}} r_L(u)^{n - 1} \langle\nabla r_L(u), Au\rangle\,d\sigma.
\end{equation}
By Lemma \ref{surf_push}, $-r_L(u)^{n - 2} \nabla r_L(u)\,d\sigma$ is the push-forward of the measure $\hat n_L\,d\mathcal H^{n - 1}$ on $\partial L$ under $x \mapsto \frac{x}{|x|}$, so pulling back the integral under this change of variables, the right hand becomes
\begin{equation}
\int_{\partial L \cap K} r_L\left(\frac{x}{|x|}\right) \left\langle \hat n_L(x), \frac{Ax}{|x|}\right\rangle\,d\mathcal H^{n - 1}.
\end{equation}
But $r_L$ is $-1$-homogeneous, so for $x \in \partial L$ we have
\begin{equation}
r_L\left(\frac{x}{|x|}\right) \left\langle \hat n_L(x), A\frac{x}{|x|}\right\rangle = \langle \hat n_L(x), Ax\rangle,
\end{equation}
and we are done.
\end{proof}

\subsection{Second proof of Theorem \ref{gen_max_int_pos}}

In the previous section we included two detailed proofs for somewhat similar theorems. In both theorems, we consider a one-parameter family of perturbations of a shape, and check how volume is affected. Since the reader may be interested in different families of perturbations, or in the general phenomenon, we provide an alternate route which works for very general families of diffeomorphisms. The drawback, however, is that in this general setting, the assumptions on the bodies are somewhat more restrictive. In particular, we obtain Theorem \ref{gen_max_int_pos}, with more restrictive assumptions, as a corollary of the following general theorem:

\begin{theorem}\label{diff_int}
Let $\varphi_t: \mathbb R^n \to \mathbb R^n$ be a family of diffeomorphisms defined on some interval $(-\epsilon, \epsilon)$ such that $\dot\varphi_t(x) = \frac{d}{dt}\varphi_t(x)$ exists everywhere and is bounded. Let $K$ be a (closed) Lipschitz domain in $\mathbb R^n$ (a set whose boundary $\partial K$ can locally be written as the zero set of a Lipschitz function), and set $K_t = \varphi_t(K)$.

Let $f$ be a bounded upper semicontinuous function defined on a neighborhood of $K$ such that
\begin{equation}
dc(f) = \{x: \text{$f$ is not continuous at $x$}\}
\end{equation}
satisfies $\mathcal H^{n - 1}(\partial K_t \cap dc(f)) = 0$ for all $t \in (-\epsilon, \epsilon)$. Then the function $\int_{K_t} f$ is differentiable on $(-\epsilon, \epsilon)$, and we have
\begin{equation}\label{lip_deriv}
\frac{d}{dt} \int_{K_t} f = \int_{\partial {K_t}} (\dot\varphi_t \cdot n_{K_t}) f\,d\mathcal H^{n - 1}
\end{equation}
where $n_{K_t}$ is the unit normal to $\partial K_t$ (which exists $\mathcal H^{n - 1}$-almost everywhere).
\end{theorem}

As a corollary, we get Theorem \ref{gen_max_int_pos}  with a slightly more restrictive assumption, namely that $\vol_{n - 1}(\partial K_t \cap \partial L) = 0$ for all $t$ in a neighborhood of $0$. Indeed, letting $f = 1_L$, which is upper semicontinuous and discontinuous precisely on $\partial L$, and letting $\varphi_t$ be the same families of transformations as before, namely translation by $tu$ for $u \in \mathbb R^n$ and multiplication by $e^{t A}$ for a traceless matrix $A$, we see that $\int_{K_t} f = \vol(\varphi_t(K) \cap L)$.

\begin{proof}[Proof of Theorem \ref{diff_int}]
Our starting point is that \eqref{lip_deriv} holds when $f$ is smooth; this is the usual formula for the derivative of an integral over a time-varying domain (see, e.g., \cite{Fl}). Since \eqref{lip_deriv} does not involve derivatives of $f$, it is easy to show that it holds for all continuous functions. Indeed, let $f$ be continuous. Given $\epsilon > 0$, let $g$ be smooth with $\|g - f\|_\infty < \epsilon$. For any $\delta$, we have
\begin{multline}\label{lip_deriv_approx}
\frac{1}{\delta}\left|\int_{K_{t + \delta}} f - \int_{K_t} f - \delta\int_{\partial {K_t}} f (\dot\varphi_t \cdot n_{K_t}) \right| \le \\ \frac{1}{\delta}\left|\int_{K_{t + \delta}} g - \int_{K_t} g - \delta\int_{\partial {K_t}} g (\dot\varphi_t \cdot n_{K_t})\right| + \frac{1}{\delta}\int_{K_{t + \delta}\triangle K_t} |f - g| + \left|\int_{\partial {K_t}} (f - g) (\dot\varphi_t \cdot n_{K_t}) \right|
\end{multline}
where $\triangle$ denotes the symmetric difference.
The first term in \eqref{lip_deriv_approx} tends to $0$ as $\delta \to 0$ since \eqref{lip_deriv} holds for the smooth function $g$. The second term is bounded by $\frac{\|g - f\|_\infty}{\delta}\vol(K_{t + \delta}\triangle K_t)$, but
\begin{equation}
\frac{1}{\delta}\vol(K_{t + \delta}\triangle K_t) = \frac{1}{\delta} \left(\int_{K_{t + \delta}} 1 - \int_{K_t} 1\right) \to \vol(K_t)',
\end{equation}
which is finite because \eqref{lip_deriv} holds for the smooth function $1$. Finally, the third term is bounded by $\|g - f\|_\infty \||\dot\varphi_t|\|_\infty \mathcal H^{n - 1}(\partial K_t)$. Hence we obtain
\begin{multline}\limsup_{\delta \to 0} \frac{1}{\delta}\left|\int_{K_{t + \delta}} f - \int_{K_t} f - \delta\int_{\partial {K_t}} f (\dot\varphi_t \cdot n_{K_t})\,d\mathcal H^{n - 1} \right| \\\le \epsilon (\vol(K_t)' + \||\dot\varphi_t|\|_\infty \mathcal H^{n - 1}(\partial K_t))
\end{multline}
for any $\epsilon$; taking the limit as $\epsilon \to 0$ shows that the limit on the RHS of \eqref{lip_deriv_approx} exists and equals $0$, i.e., \eqref{lip_deriv} holds for $f$.

We now consider the case where $f$ is upper semicontinuous. By Baire's characterization theorem \cite[\S 42.1]{H}, $f$ may be written as the pointwise limit of a decreasing sequence of continuous functions $f_n$  (one such sequence is the sup-convolution of $f$ with an appropriate sequence of Lipschitz kernels). Let $\psi_t^{(+)} = \max(\dot\varphi_t \cdot n_{K_t}, 0)$, $\psi_t^{(-)} = \min(\dot\varphi_t \cdot n_{K_t}, 0)$, and define the functions
\begin{align}\alpha_n(t) &= \int_{K_t} f_n & \beta_n^{(\pm)}(t) &= \alpha_n'(t) = \int_{\partial K_t} f_n \psi_t^{(\pm)} \\
\alpha(t) &= \int_{K_t} f & \beta^{(\pm)} (t) &= \int_{\partial {K_t}} f \psi_t^{(\pm)}.
\end{align}
Note that $\beta_n^{(+)}(t) + \beta_n^{(-)}(t) = \int_{\partial K_t} f_n (\dot\varphi_t \cdot n_{K_t}) = \alpha_n'(t)$. Let $\beta(t) = \beta^{(+)}(t) + \beta^{(-)}(t)  = \int_{\partial K_t} f (\dot\varphi_t \cdot n_{K_t})$; our goal is to show that $\alpha'(t) = \beta(t)$. By a standard theorem (essentially, the fact that $C^1[0, 1]$ is complete), it's sufficient to show that $\alpha_n, \alpha, \beta_n, \beta$ are continuous in $t$ and that $\alpha_n \to \alpha$, $\beta_n^{(\pm)} \to \beta^{(\pm)}$ uniformly on compact subsets of $(-\epsilon, \epsilon)$.

Assuming all the functions are continuous, the uniform convergence of $\int_{K_t} f_n$ to $\int_{K_t} f$ and of $\int_{\partial {K_t}} f_n \psi_t^{(\pm)}$ to $\int_{\partial {K_t}} f \psi_t^{(\pm)}$ on compact subsets follows immediately by our choice of $f_n$: indeed, for all three sequences, pointwise convergence follows from the monotone convergence theorem (note that for every $t$, $f_n \psi_t^{(-)}$ is monotone increasing in $n$, while $f_n \psi_t^{(-)}$ is monotone decreasing), and pointwise convergence implies uniform convergence on compact subsets by monotonicity and Dini's theorem.

So we have reduced to the following claim: given $K$, $\varphi_t$, $K_t = \varphi_t(K)$ as above, and $f$ bounded such that the discontinuity set $dc(f)$ of $f$ intersects $K_t$ in a $\mathcal H^{n - 1}$-null set for all $t$, the functions $\alpha(t) = \int_{K_t} f$ and $\beta^{(\pm)}(t) = \int_{\partial {K_t}} f \psi_t^{(\pm)}$ are continuous in $t$.

For $\alpha$ this is easy: $|\alpha(t + \delta) - \alpha(t)| \le \|f\|_\infty \vol(K_{t + \delta} \triangle K_t)$, and we have already seen that $\vol(K_{t + \delta} \triangle K_t) = O(\delta)$. For $\beta^{(\pm)}$, we use the change of variables formula for rectifiable sets to write
\begin{equation}\beta(t) = \int_{\partial K} |J_{\partial K}(\varphi_t)| (f \circ \varphi_t) (\psi_t^{(\pm)} \circ \varphi_t)
\end{equation}
where $J_{\partial K}$ is the $(n - 1)$-dimensional Jacobian of $\varphi_t$ restricted to $\partial K$. The details of the formula don't matter: what we need is simply that the discontinuity set of $f \circ \varphi_t$, which is $\varphi_t^{-1}(dc(f))$, has $\mathcal H^{n - 1}$-null intersection with $\partial K$, and the other terms in the integrand are continuous in $t$; hence, writing the integrand as $m(t, x)$, we have that for any $t$ and any given sequence $t_n \to t$, $m(t_n, x) \to m(t, x)$ almost everywhere, so by dominated convergence, $\int_{\partial K} m(t_n, x) \to \int_{\partial K} m(t, x)$, i.e., $\beta^{(\pm)}(t_n) \to \beta^{(\pm)}(t)$. Hence $\beta^{(\pm)}$ is continuous, which concludes the proof.
\end{proof}

\subsection{Further remarks on general maximal intersection position}
In this concluding subsection we make some remarks on general maximal intersection position and its relationship to the general John positions of Section \ref{section:pos_john}, following in the footsteps of Artstein-Avidan and Katzin \cite{AK}.

Artstein-Avidan and Katzin pointed out that the maximal intersection position of a convex body $K$ and $B^n_2$ is unique if a certain variant of the strong (B)-property holds for the uniform measure $\mu_{B^n_2}$ on $B^n_2$. A centrally symmetric measure $\mu$ is said to have the strong (B)-property if for every centrally symmetric convex body $K$ and every diagonal matrix $D$, the function $t \mapsto \mu(e^{tD} K)$ is log-concave. Suppose $\mu_{B^n_2}$ has the strong (B)-property, and suppose that $B_2^n$ and $A B_2^n$ are distinct maximizers of $\{\vol(K \cap \mathcal E): \text{$\mathcal E$ an ellipsoid of volume $1$}\}$. By choosing an appropriate basis we may write $A B^n_2 = e^{-\Lambda} B^n_2$ for some traceless diagonal matrix $\Lambda$. Consider the function
\begin{equation}
f(t) = \vol(K \cap e^{-\Lambda t} B_2^n) = \vol(e^{\Lambda t} K \cap B_2^n) = \mu_{B^n_2}(e^{\Lambda t} K).
\end{equation}

By assumption, $f$ is log-concave and attains its maximum at $0, 1$, so it must be constant on $[0, 1]$. This would yield a contradiction if we make the reasonable further assumption that the only equality cases in the inequality
\begin{equation}
\vol(e^{\frac{\Lambda}{2}} K \cap B_2^n)^2 \ge \vol(K \cap B_2^n) \vol(e^\Lambda K \cap B_2^n)
\end{equation}
are the trivial ones, namely when $K \subset B_2^n$ or $B_2^n \subset K$. We call this the ``double-strength'' (B)-property.

Cordero-Erausquin and Rotem have recently shown that the strong (B)-property holds for rotationally-invariant log-concave measures \cite{CR}, which covers in particular the uniform measure on $B^n_2$ and hence shows that the maximal intersection position of $K$ with respect to $B^n_2$ is in fact unique. (This of course continues to hold when $B^n_2$ is replaced by $r B^n_2$, and more generally when $B^n_2$ is replaced by any ellipsoid, which can be transformed into $r B^n_2$ by an affine transformation whose linear component lies in $SL_n$.) However, they do not examine the equality cases, and it seems difficult to extract them for the uniform measure $B^n_2$ from their method, which directly proves the strong (B)-property for smooth rotationally invariant log-concave measures and then obtains the property for nonsmooth measures by approximation.

In the general setting of two bodies $K, L$, even the ``double-strength'' (B)-property for one of the bodies would not suffice to obtain uniqueness of the maximal intersection position, because two positions of $L$ might not be related by a positive-definite matrix. In fact, we already know that the maximal intersection position of two bodies cannot be unique in general, because the position of maximal volume is not unique and this is a special case of maximal intersection position. However, along the lines of the positive John position we introduced in \S \ref{section:pos_john}, we are led to suggest a definition of positive maximal intersection position, where the images of $K$ intersected with $L$ vary only over positive matrices:

\begin{definition}[Positive maximal intersection position] Let $K, L \subset \mathbb R^n$ be convex bodies. We say that $K, L$ are in positive maximal intersection position with respect to each other if for every positive-definite symmetric matrix $A \in SL_n$ and $z \in \mathbb R^n$ and $L' = AL + z$, we have $\vol(K \cap L') \le \vol(K \cap L)$; clearly, this definition is symmetric with respect to an interchange of $K$ and $L$.
\end{definition}

With this definition, the same argument yields that positive maximal intersection position is unique if the uniform measure on $L$ satisfies the double-strength (B)-property with respect to the body $K$.

We may state and prove a version of Theorem \ref{gen_max_int_pos} in the setting of \textit{positive} maximal intersection position, with symmetrization required to get a decomposition of the identity, as in Theorem \ref{pos_john_contact}. The proof of the following theorem is identical to the proof of Theorem \ref{gen_max_int_pos}, the sole difference being that the orthogonality to {\em symmetric} traceless matrices forces us to consider the symmetric part of $n_K(x) \otimes x$:

\begin{theorem}
	 Let $K, L \subset \mathbb R^n$ be convex bodies, and suppose that $K, L$ are in maximal intersection position and that $\vol_{n-1}(\partial K\cap\partial L) = 0$.  Then we have
	\begin{align}
 \int_{K \cap \partial L} \hat n_L(x)\,d\mathcal H^{n - 1}           &= 0            \\
	 \int_{\partial K \cap L} (n_K(x) \otimes x + x \otimes n_K(x))\,d\mathcal H^{n - 1}(x)   &\propto I_n  .
	\end{align}
\end{theorem}

We can also consider maximal intersection position from the set-valued analysis perspective introduced in \S \ref{subsection:set_valued}. Here, we observe that the continuity of $\vol(K \cap AL)$ in $A$ means that one can apply the maximum theorem, showing that the map
\begin{equation}\label{int_vol_fun}
f^*: r \mapsto \max \{\vol(K \cap (r A L + z): \det(A) = 1, z \in \mathbb R^n\}
\end{equation}
is continuous in $r$, and the correspondence
\begin{equation}
C^*: r \mapsto \{(A, z): \vol(K \cap (A L + z) = f^*(r)\}
\end{equation}
is upper hemicontinuous. In the case $L = B^n_2$ this was implicitly observed by Artstein-Avidan and Katzin, though they were somewhat cavalier about the possible non-uniqueness of the maximal intersection ellipsoids (see \cite[Lemmas 2.3, 2.4]{AK}). 

Note that in the setting of generic $K, L$, the maximum theorem can be used both for general maximal intersection position, which maximizes over all affine images of $L$ with given volume, and for positive maximal intersection position, which maximizes only over positive images. In the latter case, we obtain a function and a correspondence of two parameters:
\begin{align}
f^*_{pos}: r, U &\mapsto \max \{\vol(K \cap (r P U L + z): P \in \mathcal P^n, z \in \mathbb R^n\,|\,\det(P) = 1\} \\
C^*_{pos}: r, U &\mapsto \{(P, z): \vol(K \cap (r P L + z) = f^*(r, U)\} \label{pos_int_vol_corr}
\end{align}
Maximizing $f^*_{pos}$ with respect to $U$ simply yields the $f^*$ of Equation \eqref{int_vol_fun}, while considering a fixed $U$ leads one to the conjecture stated above, namely that the correspondence $C^*_{pos}$ of \eqref{pos_int_vol_corr} is actually single-valued. In any case, the maximum theorem guarantees that if $K$ is in positive John position with respect to $L$, any sequence $L_n$ of positive images of $L$ having maximal intersection with $K$ such that $\vol(L_n) \to \vol(L)$ necessarily converges to $L$, as Artstein-Avidan and Katzin showed in the case $L = B^n_2$.

Finally, Artstein-Avidan and Katzin showed that if $K$ is in John position, the isotropic measures on $S^{n - 1}$ obtained by taking the maximal intersection position of $K$ with volume restriction approaching the volume of the ball, converge to an isotropic measure on $S^{n - 1}$ \cite[Theorem 1.5]{AK}. In other words, the isotropic measure guaranteed by John's theorem can be thought of as a limit of maximal intersection measures.

We note that the same exact proof applies in the general setting to yield the following generalization of their result:

\begin{theorem}
	\label{thm:measure-limit-john}Let $K, L \in \mathcal K^n$ such that $L$ is in a position of maximal volume with respect to $K$. For every $r > 1$, let
	\begin{equation}
	T_r L + z_r \in \argmax \{\vol((TL + z) \cap K): \det(T) = r, z \in \mathbb R^n\}
	\end{equation}
	be an image of $L$ of volume $r \vol(L)$ having maximal intersection with $K$, and denote by $\mu_{r}$ the uniform probability measure on $\partial K\backslash (T_r L + z_r)$. Also suppose $\vol_{n-1}(\partial K\cap\partial(T_r L + z_r))=0$ for $r$ sufficiently close to $1$. Then there exists a sequence $r_{j}\searrow 1$ such that the sequence of measures $\mu_{r_{j}}$ weakly converges to a measure $\mu$ supported on $\partial K\cap \partial L$, and any such limiting measure satisfies $\int_{\partial K} x \otimes n_K(x)\, d\mu(x) = \frac{I_n}{n}$.
\end{theorem}

\bibliographystyle{amsplain}

\end{document}